\documentclass[10pt,english,letter, reqno]{amsart}
\usepackage{amsthm}
\usepackage[T1]{fontenc}

\usepackage{amssymb,url,xspace,amsthm,adjustbox}

\usepackage[verbose,colorlinks,pagebackref]{hyperref}
\usepackage[alphabetic,nobysame]{amsrefs}

\usepackage{mathtools}
\usepackage{graphicx}

\newcommand{\fg}{{\mathfrak{g}}}

\newcommand{\bspm}{\left(\begin{smallmatrix}}
\newcommand{\espm}{\end{smallmatrix}\right)}

\usepackage{datetime}
\usepackage[T1]{fontenc}
\usepackage{chngcntr}

\theoremstyle{plain}
      \newtheorem{theorem}{Theorem}[section]
      \newtheorem{proposition}[theorem]{Proposition}
      
      \newtheorem{lemma}[theorem]{Lemma}
      \newtheorem{corollary}[theorem]{Corollary}
      \newtheorem{conjecture}[theorem]{Conjecture}

      \theoremstyle{definition}
      \newtheorem{definition}{Definition}

      \newtheorem{remark}[theorem]{Remark}
      \newtheorem{example}[theorem]{Example}

      \newtheorem*{conjecture*}{Conjecture}

\counterwithin{table}{subsection}

\usepackage{mathtools}
\usepackage{graphicx}

\usepackage{tikz}
\usetikzlibrary{matrix}
\usepackage{stmaryrd}
\usepackage{scrextend}
\usepackage{colortbl}
\usepackage{tabularx}
\usepackage{nicematrix}
\usetikzlibrary{arrows.meta}

\newcommand{\CC}{{\mathbb{C}}}

\DeclareMathOperator{\GL}{GL}
\DeclareMathOperator{\Sym}{Sym}

\DeclareMathOperator{\SL}{SL}

\DeclareMathOperator{\Sp}{Sp}

\DeclareMathOperator{\SO}{SO}

\newcommand{\m}{{\mathfrak{m}}}
\newcommand{\Lgroup}[1]{{\hskip-2 pt \,^L\hskip-1pt{#1}}}
\newcommand{\dualgroup}[1]{{\widehat{#1}}}

\newcommand{\g}{{\mathfrak{g}}}

\newcommand\Sl{\mathfrak{sl}}
\newcommand\U{\mathfrak{u}}
\newcommand\T{\mathfrak{t}}
\newcommand\lev{\mathfrak{l}} 
\newcommand\p{\mathfrak{p}} 
\newcommand\bo{\mathfrak{b}} 

\DeclareMathOperator{\Irrep}{Irrep}

\DeclareMathOperator{\codim}{codim}

\DeclareMathOperator{\Ad}{Ad}
\DeclareMathOperator{\ad}{ad}

\newcommand{\abs}[1]{{\vert #1 \vert}}
\newcommand{\ceq}{{\, :=\, }}
\newcommand{\tq}{{\ \vert\ }}
\newcommand{\iso}{{\ \cong\ }}

\DeclareMathOperator{\diag}{diag}

\renewcommand{\m}{{\mathfrak{m}}}

\newcommand{\Perv}{\mathsf{Per}}

\newcommand{\Loc}{\mathsf{Loc}}
\newcommand{\Rep}{\mathsf{Rep}}

\newcommand{\Evs}{\operatorname{\mathsf{E}\hskip-1pt\mathsf{v}\hskip-1pt\mathsf{s}}}

\newcommand{\NEvs}{\operatorname{\mathsf{N}\hskip-1pt\mathsf{E}\hskip-1pt\mathsf{v}\hskip-1pt\mathsf{s}}}

\newcommand{\IC}{{\mathcal{I\hskip-1pt C}}}

\newcommand{\RPhi}{{\mathsf{R}\hskip-0.5pt\Phi}}

\makeatletter
\newcommand{\labitem}[2]{
\def\@itemlabel{\textbf{#1}}
\item
\def\@currentlabel{#1}\label{#2}}
\makeatother

\newcommand{\1}{{\mathbbm{1}}}

\newcommand{\C}{\mathbb{C}}

\newcommand{\Aut}{\text{Aut}}

\newcommand{\Hom}{\text{Hom}}

\renewcommand{\Im}{\operatorname{Im}}

\newcommand{\ABV}{{\mbox{\raisebox{1pt}{\scalebox{0.5}{$\mathrm{ABV}$}}}}}

\newcommand{\ABVpure}{{\mbox{\raisebox{1pt}{\scalebox{0.5}{$\mathrm{ABV, pure}$}}}}}

\newcommand{\pure}{{\mbox{\raisebox{1pt}{\scalebox{0.5}{$\mathrm{pure}$}}}}}

\newcommand{\Fr}{{\rm {Fr}}}

     \newcommand{\ra}{\rightarrow}

\newcommand{\wh}{\widehat}

\newcommand{\wpair}[1]{\left\{{#1}\right\}}

\newcommand{\bpm}{\begin{pmatrix}}
\newcommand{\epm}{\end{pmatrix}}

\newcommand{\frob}{\mathfrak{fr}}

\usepackage{float}

\author[C. Cunningham]{Clifton Cunningham}
\address{Department of Mathematics and Statistics, University of Calgary, 
2500 University Drive NW, 
Calgary, Alberta, 
T2N 1N4, 
Canada}
\email{clifton@automorphic.ca}
\thanks{Clifton Cunningham's research is supported by NSERC Discovery Grant RGPIN-2020-05220. He is grateful to the \href{www.fields.utoronto.ca}{Fields Institute for Research in Mathematical Sciences} where some of this research was conducted.}

\author[S. Dijols]{Sarah Dijols}
\address{University of British Columbia, PIMS, Office 4132, Earth Sciences Building, 2207 Main Mall, Vancouver, BC V6T 1Z4}
\email{sarah.dijols@math.ubc.ca}
\thanks{Sarah Dijols's research is supported by the Pacific Institute for Mathematical Sciences.}

\author[A. Fiori]{Andrew Fiori}
\address{Department of Mathematics and Statistics, University of Lethbridge,
4401 University Drive,
Lethbridge, Alberta,
T1K 3M4,
Canada}
\email{andrew.fiori@uleth.ca}
\thanks{Andrew Fiori thanks and acknowledges the University of Lethbridge for their financial support as well as the support of NSERC Discovery Grant RGPIN-2020-05316.}

\author[Q. Zhang]{Qing Zhang}

\address{School of Mathematics and Statistics, Huazhong University of Science and Technology, Wuhan, 430074, China}
\email{qingzh@hust.edu.cn}
\thanks{Qing Zhang's research is supported by NSFC grant 12371010.}

\makeatletter
\let\@wraptoccontribs\wraptoccontribs
\makeatother

\usepackage{datetime}
\usepackage[T1]{fontenc}
\usepackage{chngcntr}

\usepackage{amssymb}
\usepackage{mathrsfs,stmaryrd}
\usepackage{yfonts, bbm}
\usepackage{enumitem}

\usepackage{hyperref}
\hypersetup{
  colorlinks   = true, 
  urlcolor     = blue, 
  linkcolor    = blue, 
  citecolor   = green 
}

\usepackage{tikz}
\usetikzlibrary{shapes,arrows,calc,matrix}
\usepackage{tikz-cd}
\usepackage{todonotes}
\usepackage{color}

\usepackage{amsmath}
\usepackage{lipsum}
\usepackage{setspace}

\counterwithin{table}{subsection}

\title[Whittaker normalization of $p$-adic ABV-packets]{Whittaker normalization of $p$-adic ABV-packets and Vogan's conjecture for tempered representations}

\date{\today}                                           

\begin{document}

\begin{abstract}
We show that ABV-packets for $p$-adic groups do not depend on the choice of a Whittaker datum, but the function from the ABV-packet to representations of the appropriate equivariant fundamental group does, and we find this dependence exactly.
We study the relation between open parameters and tempered parameters and Arthur parameters and generic representations. 
We state a genericity conjecture for ABV-packets and prove this conjecture for quasi-split classical groups and their pure inner forms.
Motivated by this we study ABV-packets for open parameters and prove that they are L-packets, and further that the function from the packet to the equivariant fundamental group given by normalized vanishing cycles coincides with the one given by the Langlands correspondence.
From this conclude Vogan's conjecture on A-packets for tempered representations: ABV-packets for tempered parameters are Arthur packets and the function from the packet to the equivariant fundamental group given by normalized vanishing cycles coincides with the one given by Arthur.
\end{abstract}

\maketitle

\setcounter{section}{-1}

\section{Introduction}

 Arthur packets are a key concept in establishing the local Langlands correspondence for classical groups, see \cites{Arthur:book, Mok:Unitary, KMSW:Unitary}.  
 ABV-packets for $p$-adic reductive groups were defined in \cite{CFMMX} and \cite{Vogan:Langlands}, inspired by \cite{ABV} which treats real groups, and are supposed to give a purely geometric construction of A-packets and also provide a generalization of Arthur packets to all connected reductive groups and all $L$-parameters, not only those of Arthur type.

We briefly recall the definition of ABV-packets for $p$-adic groups. For unexplained notations and more details, see Section \ref{sec:bacground}. Let $F$ be a non-archimedean local field with Weil group $W_F$ and let $G$ be a connected reductive group over $F$ with $L$-group $\Lgroup{G}$. 
Let $\Phi(G)$ be the set of equivalence classes of Langlands parameters of $G$. Given an infinitesimal parameter $\lambda: W_F\to \Lgroup{G}$, consider the Vogan variety $V_\lambda=\wpair{x\in \widehat{\mathfrak{g}}: \Ad(\lambda(w))(x)=\abs{w}_F\, x,\ \forall w\in W_F}$ and let $H_\lambda$ be the centralizer of $\lambda$ in $\widehat{G}$. Then it is known that the set $\{\phi\in \Phi(G): \lambda_\phi=\lambda\}$ is parametrized by $V_\lambda\sslash H_\lambda$, the orbits of $H_\lambda$ action on $V_\lambda$ by conjugation.  Vogan's reformulation of the local Langlands correspondence gives a bijection 
$$\mathcal{P}_{\mathfrak{w}}: \Pi^\pure_\lambda(G) \to \Perv_{H_{\lambda}}(V_\lambda)^{\mathrm{simple}}_{/\mathrm{iso}},$$
where 
\begin{itemize}
    \item $\mathfrak{w}$ is a choice of Whittaker data,
    \item $\Pi^\pure_\lambda(G)$ is the set of irreducible representations of $G(F)$, and its pure inner forms, with infinitesimal parameter $\lambda$, and
    \item $\Perv_{H_{\lambda}}(V_\lambda)^{\mathrm{simple}}_{/\mathrm{iso}} $ is the set of isomorphism classes of simple objects in the category of $H_\lambda$-equivariant perverse sheaves on $V_\lambda$.
\end{itemize}   
The key tool in the construction of ABV-packets is the functor 
\[
\NEvs_{\phi}: \Perv_{H_{\lambda}}(V_\lambda)\to  \Rep(A^\ABV_{\phi}),
\]
where $\phi$ is any Langlands parameter with $\lambda_\phi=\lambda$, defined in \cite{CFMMX}. 
Using this, one defines
$$\Pi_{\phi,\mathfrak{w}}^{\ABVpure}(G)=\{\pi\in \Pi^\pure_\lambda(G)\tq \NEvs_{\phi}(\mathcal{P}_{\mathfrak{w}}(\pi))\ne 0  \}$$
and
$$\Pi_{\phi,\mathfrak{w}}^{\ABV}(G)=\{\pi\in \Pi_\lambda(G)\tq \NEvs_{\phi}(\mathcal{P}_{\mathfrak{w}}(\pi))\ne 0  \}.$$
Moreover, there is a natural map $$\NEvs_{\phi}\circ \mathcal{P}_{\mathfrak{w}}:\Pi_{\phi,\mathfrak{w}}^{\ABVpure}(G)\to  \Rep(A^\ABV_{\phi}). $$
Vogan's conjecture says that if $G$ is a group for which we know Arthur's conjecture, such as a quasisplit classical group or an inner twist of such, and if $\phi$ is of Arthur type so $\phi=\phi_\psi$ for an Arthur parameter $\psi$, then one should have 
\begin{equation}
    \Pi_{\phi,\mathfrak{w}}^{\ABVpure}(G)=\Pi^\pure_\psi(G)
\end{equation} 
and the map 
\[
\NEvs_{\phi}\circ \mathcal{P}_{\mathfrak{w}}:\Pi_{\phi,\mathfrak{w}}^{\ABVpure}(G)\to  \Rep(A^\ABV_{\phi})
\]
should agree with the pure packet version of Arthur's map 
\[
\begin{array}{rl}
\Pi^\pure_\psi(G) & \to \widehat{A_\psi} \\
\pi &\mapsto \langle \pi ,\ \rangle_\psi 
\end{array}
\]
under the identification $A^\ABV_{\phi}\cong A_\psi$.  Such an expectation has been established for $\GL(n)$ in \cites{CR, CR2}. 
In this paper we prove this for tempered parameters.

Unlike the trace formula approach in \cites{Arthur:book, Mok:Unitary, KMSW:Unitary}, which established the local Langlands correspondence by proving the existence of Arthur packets, the definition of ABV-packets as explained above uses the local Langlands correspondence {\it a priori}.   
But the local Langlands correspondence depends on a choice of Whittaker normalization, which is not unique in general. Our first goal in this paper is to study how ABV-packets depend on the Whittaker normalization in the local Langlands correspondence. 
 Our first result is 
 \begin{theorem}[Theorem \ref{theorem:CoW}]
Assume the standard desiderata of the local Langlands correspondence for $G$, as articulated in Section~\ref{sec:desiderata}. 
For any Langlands parameter $\phi$ for $G$, the pure ABV-packet $\Pi_{\phi,\mathfrak{w}}^{\ABVpure}(G)$ is independent of the choice of the Whittaker data $\mathfrak{w}$ and the map $\NEvs_{\phi}\circ \mathcal{P}_{\mathfrak{w}}:\Pi_{\phi,\mathfrak{w}}^{\ABVpure}(G)\to  \Rep(A^\ABV_{\phi})$ depends on $\mathfrak{w}$ in a uniform way. See Theorem \ref{theorem:CoW} for a precise statement. 
 \end{theorem}
 \noindent This result is parallel with how L-packets depending on Whittaker data $\mathfrak{w}$, in the sense that as sets these also do not depend on Whittaker data but the function to characters of the appropriate equivariant fundamental group does. 
%
Indeed, we show that $\NEvs_\phi\circ \mathcal{P}_\mathfrak{w}$ is compatible with the function 
\[
J(\mathfrak{w}) : \Pi^\pure_\psi(G) \to \Rep(A_\phi)
\]
from the Langlands correspondence by showing that the following diagram commutes; see Theorem~\ref{theorem:compatibility}.
\[
\begin{tikzcd}
\Pi^\pure_\phi(G) \arrow[>->]{d}  \arrow{rr}{J(\mathfrak{w})} && \Rep(A_\phi) \arrow[>->]{d}\\
\Pi^{\ABVpure}_{\phi}(G) \arrow{rr}{\NEvs_\phi \circ\mathcal{P}_{\mathfrak{w}}} && \Rep(A_\phi^\ABV)
\end{tikzcd}
\]
Note that here we use the \emph{pure} L-packet $\Pi^\pure_\psi(G)$ for $\phi$, meaning the union of the L-packets for $G(F)$ and its pure inner forms.

One starting point of the definition of pure ABV-packets is the following fact: given an infinitesimal parameter $\lambda: W_F:\to \Lgroup{G}$, then the set $\wpair{\phi\in \Phi(G)\tq \lambda_\phi=\lambda}$ is parametrized by the space $V_\lambda\sslash H_{\lambda}$ of $H_\lambda$-orbits on $V_\lambda$ under the conjugation action. Although the structure $V_\lambda\sslash H_\lambda$ can be very complicated, it contains two special orbits: the open orbit and the closed orbit. It turns out that the L-packets for these two Langlands parameters have very special properties, as observed in the case when $G=G_2$ in \cite{CFZ:unipotent}.
%
A Langlands parameter $\phi$ is said to be open if its $H_\lambda$-orbit is open in $V_\lambda$; see Definition~\ref{def:open}.
Our next goal in this paper is to investigate the properties of L-packets and ABV-packets for open Langlands parameters. Our main result on this is
\begin{theorem}[See Proposition~\ref{prop: temper=arthur+open}, Proposition~\ref{prop:openL}, Corollary~\ref{corollary: GI}, Corollary~\ref{corollary: classical}]\label{theorem-open-intro}
Let $G$ be a connected reductive group over $F$. Then
\begin{enumerate}
\item An $L$-parameter is open and of Arthur type if and only if it is tempered.
\item An $L$-parameter $\phi$ for $G$ is open if and only if $L(s,\phi,\Ad)$ is regular at $s=1$. 
\item If $G$ is a quasi-split classical group then $\phi$ is open if and only if the L-packet $\Pi_\phi(G)$ contains a generic representation.
\item If $G$ is a pure inner form of a quasi-split classical group, then $\phi$ is open if and only if $\Pi_\phi^{\ABV\pure}(G)$ contains a generic representation.
In this case the generic representation is of the quasi-split group.
\end{enumerate}
\end{theorem}

In fact, Corollary~\ref{corollary: GI} is the following statement: if $G$ is quasi-split classical group then $\phi$ is open if and only if $\Pi_\phi^\pure(G)$ contains a generic representation. This is equivalent to Theorem~\ref{theorem-open-intro} (3) because any generic representation appearing in $\Pi_\phi^\pure(G)$ is necessarily a representation of the quasi-split group $G(F)$.

In Proposition~\ref{prop: temper=arthur+open} we prove that tempered parameters are precisely those that are open and also of Arthur type.
In Appendix~\ref{sec:discrete} we give a second proof that tempered parameters are open, 
using a different strategy than that used to prove Proposition~\ref{prop: temper=arthur+open}.

Turning to genericity,  we prove that generic parameters are open for quasi-split classical groups and their pure inner forms and conjecture that this is true in complete generality; see   Conjecture~\ref{conj: GP2} and Corollary~\ref{corollary: GI}.
Conjecture~\ref{conj: genericabv} proposes that the ABV-packet for a Langlands parameter $\phi$ contains a generic representation if and only if the parameter $\phi$ is open.
Theorem~\ref{theorem:maingeneric} proves Conjecture~\ref{conj: genericabv} for quasi-split classical groups and their pure inner forms.

Motivated by Conjecture~\ref{conj: genericabv}, we study ABV-packets for open Langlands parameters in full generality.
In Theorem~\ref{theorem:mainopen} we prove that if $\phi$ is open then the ABV-packet attached to $\phi$ is the pure L-packet attached to $\phi$, and the parametrization of these packets by the enhanced Langlands correspondence coincides with the one built from vanishing cycles in \cite{CFMMX}.
\[
\begin{tikzcd}
\Pi^\pure_\phi(G) \arrow[equal]{d}  \arrow{rr}{J(\mathfrak{w})} && \Rep(A_\phi) \arrow[equal]{d}\\
\Pi^{\ABVpure}_{\phi}(G) \arrow{rr}{\NEvs_\phi \circ\mathcal{P}_{\mathfrak{w}}} && \Rep(A_\phi^\ABV).
\end{tikzcd}
\]
From this we deduce Vogan's conjecture on A-packets of tempered representations of p-adic groups, Corollary~\ref{cor:vogantempered}: 
if $\psi$ is a tempered Arthur parameter, then the pure A-packet $\Pi_\psi^\pure(G)$ is the ABV-packet $\Pi^{\ABVpure}_\phi(G)$ and Arthur's map $A(\mathfrak{w})$ to the component group $A_\psi$ coincides with the map $\NEvs_{\phi_\psi}\mathcal{P}_{\mathfrak{w}}$ built from vanishing cycles.
\[
\begin{tikzcd}
\Pi^\pure_\psi(G) \arrow[equal]{d}  \arrow{rr}{A(\mathfrak{w})} && \Rep(A_\psi) \arrow[equal]{d}\\
\Pi^{\ABVpure}_{\phi_\psi}(G) \arrow{rr}{\NEvs_{\phi_\psi} \circ\mathcal{P}_{\mathfrak{w}}} && \Rep(A^\ABV_{\phi_\psi})
\end{tikzcd}
\]
That $A_\phi$, $A_\psi$ and $A^\ABV_\phi$ are all equivariant fundamental groups is explained in \cite{CFMMX}*{Sections 4.7, 6.7 and 7.9}.

Finally, as a consequence of the relationship between tempered parameters and open parameters, we prove the Vogan's conjecture for tempered parameters:
\begin{theorem}[Corollary \ref{cor:vogantempered}]
Let $G$ be a connected reductive group over $F$ for which Arthur's conjectures are true. 
If $\phi$ be a tempered $L$-parameter for $G$ then 
$$\Pi_\phi^{\ABV}(G)=\Pi_\psi(G),$$
where $\psi$ is the Arthur parameter corresponding to $\phi$.
See Corollary \ref{cor:vogantempered}, which is a stronger result, for more precision. 
\end{theorem}

\subsection{Relation to other works}

This paper is an expanded treatment of part of \cite{CDFZ:genericity}.


The connection between genericity and openness in Theorem~\ref{theorem-open-intro} (3) seems to be known for experts. In \cites{Reeder, Reedergeneric} it is conjectured that a unipotent representation is generic if and only if that its $L$-parameter is in the dense orbit and it corresponds to the trivial representation of the component group of the $L$-parameter. 
In the general case, it also appeared as a conjecture in \cite{Solleveld}*{Conjecture B}. Part (4) of Theorem \ref{theorem-open-intro} is inspired by the Enhanced Shahidi conjecture, see \cite{LLS}, which says that an Arthur packet $\Pi_\psi$ contains a generic representation if and only if $\psi$ is tempered. Modulo Vogan's conjecture on Arthur packets, Theorem \ref{theorem-open-intro} (4) is the ABV-packet counterpart and a generalization of the Enhanced Shahidi conjecture, in view of Theorem \ref{theorem-open-intro} (1).

Referring to \cite{CFZ:cubic} and \cite{CDFZ:genericity}, Solleveld conjectured that generic Langlands parameters are open \cite{Solleveld}*{Conjecture B}. 
Our work in this paper proves \cite{Solleveld}*{Conjecture B} for quasi-split classical groups and their pure inner forms.
In \cite{CDFZ:genericity} we show that this conjecture is true for any group satisfying the p-adic analogue of the Kazhdan-Lusztig hypothesis.

    Proposition~\ref{prop:openL} was established in \cite{DHKM}*{Proposition 6.10} after our preprint \cite{CDFZ:genericity} was posted on arxiv.

The ABV-packets in this paper were introduced in \cite{CFMMX} \cite{ABV}, strongly influenced by \cite{Vogan:Langlands} and \cite{ABV}, hence the naming.

We wish to bring the reader's attention to a change in notation and language that we have made here compared to our earlier work. 
\begin{enumerate}
\item 
In this paper we write
$\Pi^{\ABV}_\phi(G)$ for what we have denoted by 
$\Pi^{\ABV}_\phi(G(F))$ in \cite{CFMMX} and our other earlier papers; we refer to this as an "ABV-packet".
\item 
In this paper we write 
$\Pi^{\ABVpure}_\phi(G)$ for what we have denoted by 
$\Pi^{\ABV}_\phi(G/F)$ in \cite{CFMMX} and our other earlier papers; we refer to this as a "pure ABV-packet".
\end{enumerate}
We make this change to harmonize our the notation and language with the following common usage:
\begin{enumerate}
\item 
$\Pi_\phi(G)$, is, of course, an "L-packet", and
\item 
$\Pi^{\pure}_\phi(G)$, is a "pure L-packet";
\end{enumerate}
and likewise harmonize with
\begin{enumerate}
\item 
$\Pi_\psi(G)$, which is, of course, an "A-packet", and
\item
$\Pi^{\pure}_\psi(G)$, which we refer to as a "pure A-packet".
\end{enumerate}

\subsection{Acknowledgments}

The first author thanks Julia Gordon, James Arthur and Tasho Kaletha for an invitation to the 2021 BIRS Workshop \emph{Basic Functions, Orbital Integrals, and Beyond Endoscopy} (\href{http://www.birs.ca/events/2021/5-day-workshops/21w5228}{Casselman birthday conference}), where some results from this paper were presented.
The second author would like to express her gratitude to Julia Gordon for discussing with her some aspects of the Appendix, as well as pointing out the work of Ranga Rao. 
We thank Maarten Solleveld for his interest in this work and for sharing his pre-print \cite{Solleveld} with us. 

\tableofcontents

\section{Background on ABV-packets}\label{sec:bacground}

Throughout this paper, unless noted otherwise, $G$ is an arbitrary connected reductive algebraic group over a $p$-adic field $F$.
We write $G^*$ for the quasi-split form of $G$.

In this section we briefly review from \cite{CFMMX} the definition of the ABV-packet $\Pi^{\ABVpure}_\phi(G)$ and the function $\Pi^{\ABVpure}_\phi(G) \to \Rep(A^\ABV_\phi)$, where $\phi$ is a Langlands parameter for $G$, paying special attention to the dependence on the local Langlands correspondence.

\subsection{Infinitesimal parameters} \label{infinitesimaldef}

Let $\Lgroup{G}=\wh G\rtimes W_F$ be the $L$-group of $G$. 
Set $W'_F\ceq W_F\times \SL_2(\C)$.
Recall that a local Langlands parameter $\phi$ is a continuous homomorphism $\phi: W'_F\to \Lgroup{G}$ such that $\phi(\Fr)$ is semi-simple, $\phi|_{\SL_2(\C)}$ is algebraic and $\phi$ commutes with $W'_F\to W_F$ and $\Lgroup{G}\to W_F$. 

Giving a local Langlands parameter $\phi$ amounts to giving a Weil-Deligne representation of $W_F$, namely, a pair $(\lambda, x)$, where $\lambda: W_F\to \Lgroup{G}$ is a group homomorphism which is continuous on $I_F$ with $\phi(\Fr)$ semi-simple, and commutes with $\Lgroup{G}\to W_F$, and $x\in \wh \fg$ is a nilpotent element such that 
$$\Ad \lambda(w)N=\abs{w}_F\, x. $$
This correspondence is given by $\phi\mapsto (\lambda_\phi, x)$, where 
\begin{equation}\label{eq-infinitesimal}
\lambda_\phi(w)\ceq \phi(w, \diag(\abs{w}_F^{1/2}, \abs{w}_F^{-1/2}))
\end{equation}
is the infinitesimal parameter of $\phi$ and $x=d\phi \left(\begin{smallmatrix}0&1\\ 0&0\end{smallmatrix}\right)$. 
For an argument of such an equivalence, see \cite{GR}*{Proposition 2.2}. 

We write $\Phi(G)$ for the set of equivalence classes of Langlands parameters for $G$ \emph{with} the relevance condition of \cite{Borel:Corvallis}*{Section 8.2 (ii)} for $G$.
Note that $\Phi(G^*)$ is now the set of equivalence classes of Langlands parameters for $G$ \emph{without} any relevance condition, since this condition is empty for quasi-split groups. 

By an {\it infinitesimal parameter} we mean a group homomorphism $\lambda: W_F\to \Lgroup{G}$ satisfying all the conditions of a Langlands parameter that is trivial on the Deligne part of the Weil-Deligne group, without imposing the relevance condition; see \cite{CFMMX}*{Section 4.1} for more precision.

\subsection{Variety of certain Langlands parameters} \label{variety}

For a fixed infinitesimal parameter, the associated ``Vogan variety'' $V_\lambda$ is  
\[
V\ceq V_\lambda:=\{x\in \widehat {\mathfrak{g}} \tq \Ad(\lambda(w))(x)=\abs{w}_F\, x, \ \forall w\in W_F\}.
\]
Note that if we fix an element $\frob \in W_F$ such that $\abs{\frob}_F=q$, then
\[
V_\lambda= \{x\in \widehat{\mathfrak{g}}^{I_F} \tq \Ad(\lambda(\frob))(x)=q x\}.
\]
We set
\[
V^* \ceq V_\lambda^*=\{y\in \widehat{\mathfrak{g}} \tq \Ad(\lambda(w))(y)=\abs{w}_F^{-1}\, y, \ \forall w\in W_F\}.
\]
We use the Killing form for  $\widehat {\mathfrak{g}}$ to define a pairing $T^*(V_\lambda) = V_\lambda \times V_\lambda^* \to \mathbb{A}^1$; this allows us to identify $V_\lambda^*$ with the dual vector space to $V_\lambda$.

Let $H_\lambda$ be the centralizer of $\lambda$ in $\dualgroup{G}$, so 
\[
H_\lambda:=\{g\in \dualgroup{G} \tq g\lambda(w)g^{-1}=\lambda(w),\ \forall w\in W_F\}.
\]
Then $H_\lambda$ acts on both $V_\lambda$ and $V_\lambda^*$ in $\dualgroup{\mathfrak{g}}$ by conjugation and both $V_\lambda$ and $V_\lambda^*$ are prehomogeneous vector spaces for this action; in particular, there are finitely many $H_\lambda$-orbits in $V_\lambda$ and $V_\lambda^*$, each with a unique open orbit and a unique closed orbit by \cite{CFMMX}*{Proposition 5.6}\\

From \cite{CFMMX}*{Proposition 4.2}, there is a bijection between the set of Langlands parameters with infinitesimal parameter $\lambda$ and the $H_\lambda$-orbits in $V_\lambda$.
Indeed, by construction, each $x\in V_\lambda$ (resp. $y\in V_\lambda^*$) uniquely determines a Langlands parameter $\phi_x$ (resp. $\phi_y$) such that $\phi_x(w,\diag(\abs{w}_F^{1/2},\abs{w}_F^{-1/2})) = \lambda(w)$ (resp. $\phi_y(w,\diag(\abs{w}_F^{-1/2},\abs{w}_F^{1/2})) = \lambda(w)$) and $\phi_x(1,e) = \exp x$ (resp. $\phi_y(1,f) = \exp y$), where $e = \left( \begin{smallmatrix} 1 & 1 \\ 0 & 1 \end{smallmatrix}\right)$ (resp. $f = \left( \begin{smallmatrix} 1 & 0 \\ 1 & 1 \end{smallmatrix}\right)$). 
We will write $x_\phi \in V_\lambda$ for the point on $V_\lambda$ corresponding to $\phi$ where $\phi(w,\diag(\abs{w}_F^{1/2},\abs{w}_F^{-1/2})) = \lambda(w)$.
The $H_\lambda$-orbit of $x_\phi\in V_\lambda$ will be denoted by $C_\phi$. 
 
For an $H_\lambda$-orbit $C$ in $V_\lambda$, pick any point $x\in C$ and denote $A_C=\pi_0(Z_{H_\lambda}(x))$, where $Z_{H_\lambda}(x)$ denotes the stabilizer of $x$ in $H_\lambda$ and $\pi_0$ denotes the component group. 
The isomorphism class of $A_C$ is independent of the choice of $x$ and is called the equivariant fundamental group of $C$. 
Since $C$ is connected, the choice of $x\in C$ determines an equivalence between the category $\Rep(A_C)$ of finite-dimensional $\ell$-adic representations of $A_C$ and the category $\Loc_H(C)$ of $H$-equivariant local systems on $C$.
For a local Langlands parameter $\phi:W'_F\to \Lgroup{G}$, denote $A_\phi=\pi_0(Z_{\dualgroup{G}}(\phi))$, which is the component group of the centralizer of $\phi$. 
A fundamental fact is $A_{C_\phi} \cong A_\phi$, see \cite{CFMMX}*{Lemma 4.4}. 

\subsection{Conormal variety to the parameter variety}

The conormal to $V_\lambda$ is
\[
\Lambda_\lambda \ceq \{ (x,y)\in V_\lambda\times V_\lambda^* \tq [x,y]=0 \},
\]
were, $[~,~]$ is the Lie bracket in $\widehat{ \mathfrak{g}}$; see \cite{CFMMX}*{Proposition 6.3.1}.
Likewise, 
\[
\Lambda_{\lambda}^* \ceq \{ (y,x)\in V_\lambda^* \times V_\lambda \tq [x,y]=0 \}
\]
may be identified with the conormal to $V_\lambda^*$.
We write $p : \Lambda \to V$ and $q : \Lambda \to V^*$ for the obvious projections and set $\Lambda_{C}\ceq p^{-1}(C)$ and $\Lambda_{D}^*\ceq q^{-1}(D)$.

Pyasetskii duality $C \mapsto C^*$ defines a bijection between the $H_\lambda$-orbits in $V_\lambda$ and the $H_\lambda$-orbits in $V_\lambda^*$ and is uniquely characterized by the following property: under the isomorphism $\Lambda_\lambda \to \Lambda^*_\lambda$ defined by $(x,y)\to (y,-x)$, the closure of $\Lambda_{C}$ in $\Lambda_\lambda$ is isomorphic to the closure of $\Lambda_{C^*}^*$ in $\Lambda^*_\lambda$.
This duality may also be characterized by passing to the regular conormal variety as follows.
Set
\[
    \Lambda_C^\mathrm{reg}=\{(x,y)\in \Lambda_C \tq y \in q(\Lambda_C)\setminus q(\bar{\Lambda}_{C'})\text{ where }C \subseteq \bar{C'} \text{ but } C \neq C' \}, 
\]
    where $q : \Lambda_\lambda \to V^*_\lambda$ is projection, and also set
$
\Lambda_\lambda^\mathrm{reg} \ceq \bigcup_C \Lambda_{C}^\mathrm{reg}
$ (this is a disjoint union, in fact).
Likewise define $\Lambda_{D}^{*,\mathrm{reg}}$ for an $H_\lambda$-orbit $D\subset V^*$. 
Then 
\[
\Lambda_{C}^\mathrm{reg} \iso \Lambda_{C^*}^{*,\mathrm{reg}}
\]
under the isomorphism $\Lambda_\lambda\iso \Lambda^*_\lambda$ given above; see \cite{CFMMX}*{Lemma 6.4.2}.
It follows that 
\[
C^* = q_\mathrm{reg}(p_\mathrm{reg}^{-1}(C)),
\]
where $p_\mathrm{reg} : \Lambda_\lambda^\mathrm{reg} \to V_\lambda$ and $q_\mathrm{reg} : \Lambda_\lambda^{*,\mathrm{reg}} \to V_\lambda^*$ are the obvious projections.

We close this Section with two simple lemmas that we will use later.

\begin{lemma}\label{lemma:dim}
For every $H_\lambda$-orbit $C$ in $V_\lambda$, $\dim \Lambda_C = \dim V_\lambda$.
\end{lemma}

\begin{proof}
In \cite{CFMMX}*{Proposition 6.3.1} we show that $\Lambda_{C}$, as defined above, is the conormal bundle to $V_\lambda$ above $C$. 
Recall that the cotangent bundle is $V_\lambda\times V_\lambda^*$ so that the conormal bundle to $C\subset V_\lambda$ at a point $x\in C$ is given by $\Lambda_{C,x} = T^*_{C,x}(V_\lambda)$  is given by $\{ y\in T^*_x(V_\lambda)  \tq y(x') =0, \ \forall x'\in T_x(C)\}$. Observe also that $\dim \Lambda_C = \dim C + \dim \Lambda_{C,x}$ for any $x\in C$. Since $\dim T_x(C) = \dim C$ and $\dim T^*_x(V_\lambda) = \dim V^*_\lambda  = \dim V_\lambda$, it follows that $\dim \Lambda_{C,x} = \codim C$. Therefore, $\dim \Lambda_C = \dim C +  \codim C = \dim V_\lambda$.
\end{proof}

\begin{lemma}\label{lemma:open}
The dual to the closed orbit $C_0$ in $V_\lambda$ is the open orbit in $V_\lambda^*$ and the dual to the open orbit $C^o$ in $V_\lambda$ is the closed orbit in $V_\lambda^*$.
\end{lemma}

\begin{proof}
First note that $\Lambda_\lambda^\mathrm{reg}$ is open in $\Lambda_\lambda$.
By \cite{CFMMX}*{Proposition 6.4.3}, $\Lambda^\mathrm{reg}_{C} \subseteq C\times C^*$ for every $H_\lambda$-orbit $C\subset V_\lambda$. Then $\dim \Lambda_{C} \leq \dim C + \dim C^*$. It now follows from Lemma~\ref{lemma:dim} that $\dim C^* \geq \codim C$. 
Now take $C= C_0$, the closed orbit in $V_\lambda$, for which $\codim C_0 = \dim V_\lambda$.
Then  $\dim C_0^* \geq \codim C_0 = \dim V_\lambda$. Since $C_0^*\subset V_\lambda^*$, it follows that $\dim C_0^* = \dim V^*_\lambda$, so $C_0^*$ an the open $H_\lambda$-orbit in $V_\lambda$, which is unique since $V_\lambda$ is a prehomogeneous vector space.
Since orbit duality is an involution, it follows that the dual to the open orbit $C^o$ in $V_\lambda$ is the closed/trivial orbit in $V^*_\lambda$.
\end{proof}

\subsection{Vanishing cycles}\label{sec:VC}

In  \cite{CFMMX}*{Section 7.10} we define a functor
\[
\NEvs_{C}: \Perv_{H_{\lambda}}(V_\lambda)\to \Loc_{H_{\lambda}}(\Lambda_{C}^{\mathrm{gen}}),
\]
for every $H_\lambda$-orbit $C$ in $V$, where  $\Lambda^\mathrm{gen}_{C}\subseteq \Lambda_\lambda$ is a connected $H_\lambda$-stable open subset defined in \cite{CFMMX}*{Section 7.9}.
Since $\Lambda^\mathrm{gen}_{C}$ is connected we may pick a base point $(x,y)\in \Lambda^\mathrm{gen}_{C}$ and set 
\[
A^\ABV_{C}\ceq \pi^{H_\lambda}_1(\Lambda^\mathrm{gen}_{C},(x,y)),
\]
the $H_\lambda$-equivariant fundamental group of $\Lambda^\mathrm{gen}_{C}$, which allows us to rewrite $\NEvs_{C}$ as a functor landing in the category $\Rep(A^\ABV_{C})$ of finite-dimensional $\ell$-adic representations of $A^\ABV_{C}$, thus defining 
\[
\NEvs_{\phi}: \Perv_{H_{\lambda}}(V_\lambda)\to \Rep\left(A^\ABV_{\phi}\right).
\]
Note the distinction between $\NEvs_{C}$ and $\NEvs_\phi$ for $x_\phi \in C$; a similar  distinction is made in \cite{CR} and \cite{CR2}.

\subsection{Whittaker data}\label{ssec:LLC}

Let $F$ be a non-archimedean local field and $G$ be a connected quasi-split reductive group over $F$. 
Our exposition mainly follows \cite{GGP}*{Section 9} and \cite{Kaletha}*{Section 4}.

 Let $B$ be a Borel subgroup of $G$ with unipotent radical $U$. The torus $T=B/U$ acts on $\Hom(U,\C^\times)$. Recall the following basic definition.
   A character $\theta: U(F)\to \C^\times$ is called \emph{generic character} if its stabilizer in $T(F)$ is the centre $Z(F)$ of $G(F)$.  
For a generic character $\theta$ of $U(F)$, a representation $\pi$ of $G(F)$ is called $\theta$-generic if $\Hom_{U(F)}(\pi,\theta)\ne 0$. 
This only depends on the $T(F)$-orbit of $\theta$. 
If $\theta$ is understood from the context, we simply say $\pi$ is generic.

The set $D$ of $T(F)$-orbits on the generic characters forms a principal homogeneous space for the abelian group 
$T_{ad}(F)/\Im(T(F))$,
where $T_{ad}(F)$ is the corresponding maximal torus of the adjoint group $G_{ad}(F)$. A \emph{Whittaker datum} $\mathfrak{w}$ for $G$ is a $G(F)$-conjugacy classes of pairs $(B,\theta)$, where $B$ is a Borel subgroup of $G$ and $\theta$ is a generic character of $U(F)$, where $U$ is the unipotent radical of $B$. 
Such a pair is a principal homogeneous space for the abelian group $G_{ad}(F)/\Im(G(F))$. 

It is known that 
\[
T_{ad}(F)/\Im(T(F)) \iso G_{ad}(F)/\Im(G(F)).
\]
Thus if we fix a Borel subgroup $B$, giving a Whittaker datum $(B',\theta')$ is equivalent to giving an element in $D$.
This means there are two ways to define Whittaker datum: one is using elements in $D$ by fixing a Borel, which is a principal homogeneous space for $T_{ad}(F)/\Im(T(F))$; and the other one is using the $G(F)$-conjugacy classes of pairs $(B, \theta)$, which is a principal homogeneous space for the group $G_{ad}(F)/\Im(G(F))$. 
The first one is easy to compute and the second one is more intrinsic because it does not need to fix a Borel in advance. 
The isomorphism $T_{ad}(F)/\Im(T(F)) 
\iso
G_{ad}(F)/\Im(G(F))$ says that the two approaches are the same.

\subsection{Change of Whittaker data}\label{sec:CoWh}

Let $\wh G_{sc}$ be the simply connected cover of $\wh G_{ad}$. 
Let $\wh Z$ (resp. $\wh Z_{sc}$) be the center of $\wh G$ (resp. $\wh G_{sc}$). 
For a finite abelian group $A$, denote by $A^D$ the group of characters of $A$. 
By \cite{Kaletha}*{Lemma 4.1}, see also \cite{Borel:Corvallis}*{Section 10}, there exists a canonical bijection
\begin{equation}\label{eq: Kaletha 4.1}
G_{ad}(F)/G(F) \to (\ker(H^1(W_F, \wh Z_{sc})\to H^1(W_F, \wh Z)))^D.
\end{equation}
Giving two Whittaker data $\mathfrak{w}$ and $\mathfrak{w}'$, let 
\[
(\mathfrak{w},\mathfrak{w}')\in G_{ad}(F)/G(F)
\]
 be the element which conjugates $\mathfrak{w} $ to $\mathfrak{w}'$. 
 By the above bijection \eqref{eq: Kaletha 4.1}, the pair $(\mathfrak{w},\mathfrak{w}')$  determines a character
\begin{equation}
\ker(H^1(W_F, \wh Z_{sc})\to H^1(W_F, \wh Z)) \to \CC^\times,
\end{equation}
also denoted by $(\mathfrak{w},\mathfrak{w}')$.

Every Langlands parameter $\phi: W_F\times \SL_2(\C)\to \Lgroup{G}=\wh G\rtimes W_F$ defines an action of $W_F$ on $\wh G$ by $w.g=\phi(w,1)\, g\, \phi(w,1)^{-1}$. 
This action of $W_F$ on $\wh G$ induces an action on $\wh G_{sc}$ and restricts to an action on $\wh G_{ad}$. 
The action of $W_F$ on $\wh Z$ (resp. on $\wh Z_{sc}$) induced by any $\phi$ is compatible with the natural action of $W_F$ on $\wh Z$ (resp. $\wh Z_{sc}$), namely the action when we define the $L$-group $\Lgroup{G}=\wh G\rtimes W_F$, see \cite{Borel:Corvallis}. 
Now consider the composition 
$$f: H^0_\phi(W_F, \wh G)\to H^0_\phi(W_F, \wh G_{ad})\to H^1(W_F, \wh Z_{sc}),$$
where $H^0_\phi(W_F, \wh G)$ denotes the elements in $\wh G$ invariant under the $W_F$-action induced by $\phi$, the first map comes from the long exact sequence associated with the short exact sequence 
$$1\to Z\to \wh G\to \wh G_{ad}\to 1,$$
and the second map comes from the long exact sequence associated with the short exact sequence
$$1\to \wh Z_{sc}\to \wh G_{sc}\to \wh G_{ad}\to 1.$$
Moreover, the composition of $H^1(W_F, \wh Z_{sc})\to H^1(W_F, \wh Z)$ with $f$ vanishes because it is the composition
$$H^0_\phi(W_F, \wh G)\to H^0_\phi(W_F, \wh G_{ad})\to H^1(W_F, \wh Z),$$
which is part of the long exact sequence associated with the short exact sequence 
$$1\to \wh Z\to \wh G\to \wh G_{ad}\to 1.$$
Thus $f$ induces a homomorphism 
\begin{equation}
H^0_\phi(W_F, \wh G)\to H^0_\phi(W_F, \wh G_{ad})\to \ker( H^1(W_F, \wh Z_{sc})\to H^1(W_F, \wh Z)).
\end{equation}
This character is trivial on $Z(\wh G)^{W_F}$. See \cite{Kaletha}*{Section 4} for more details.
Note that 
\[
H^0_\phi(W_F, \wh G) = \wpair{g\in \wh G \tq  \phi(w,1)g=g\phi(w,1), \forall w\in W_F}
\]
which contains
\[
Z_{\wh G}(\phi) \ = \wpair{g\in \wh G\tq \phi(w,x)g=g\phi(w,x), \forall w\in W_F, \forall x\in \SL_2(\CC)}.
\]
The composition
\begin{equation}\label{eqn:CoW}
\begin{tikzcd}
Z_{\wh G}(\phi) \arrow[>->]{r} & H^0_\phi(W_F, \wh G) \arrow{r}{f} & \ker\left(H^1(W_F, \wh Z_{sc}) \to H^1(W_F, \wh Z)\right) \arrow{r}{(\mathfrak{w},\mathfrak{w}')} & \CC^\times
\end{tikzcd}
\end{equation}
is trivial on $Z_{\wh G}(\phi)^\circ$ and therefore determines a character of the component group $A_\phi = \pi_0(Z_{\wh G}(\phi))$. 
We denote this character by 
\begin{equation}\label{eqn:twisting}
\varrho_{\mathfrak w, \mathfrak w'}^\phi: A_\phi\to \C^\times.
\end{equation}

\begin{lemma}\label{lemma:PIFvarrho}
    The image of $\varrho^\phi_{\mathfrak{w},\mathfrak{w}'}$ under $\widehat{A_\phi} \to H^1(F,G^*)$ is trivial.
\end{lemma}

\begin{proof}
    Because the character in Equation~\ref{eqn:CoW} is trivial on $Z_{\wh G}(\phi)^\circ$, the pure inner form attached to the character $\varrho_{\mathfrak w, \mathfrak w'}^\phi$ of $A_\phi$ is the quasi-split form $G^*$ of $G$.
\end{proof}

In the special case when $\phi$ is a closed Langlands parameter, we may identify $\phi$ with its infinitesimal parameter $\lambda$. 
Now $\varrho^\phi_{\mathfrak{w},\mathfrak{w}'}$ is a character of $A_\lambda = H_\lambda/H_\lambda^\circ = \pi_0(H_\lambda)$ denoted by
\begin{equation}\label{eqn:twistinglambda}
\varrho^\lambda_{\mathfrak w, \mathfrak w'}: A_\lambda \to \C^\times.
\end{equation}
This case is universal for all Langlands parameters with infinitesimal parameter $\lambda$ in the following sense.
Since the image of $\lambda$ is contained in the image of $\phi$, we have a containment $Z_{\dualgroup{G}}(\phi) \subseteq Z_{\dualgroup{G}}(\lambda)$ which induces
\[
\iota_{\lambda \leq \phi} : A_\phi \to A_\lambda
\]
Then, by construction, 
\[
\varrho^\phi_{\mathfrak w, \mathfrak w'} = \varrho^\lambda_{\mathfrak w, \mathfrak w'} \circ \iota_{\lambda \leq \phi}.
\]

\subsection{Local Langlands correspondence}\label{sec:desiderata}

In this section we briefly review certain basic desiderata of the local Langlands correspondence in a form proposed by Vogan \cite{Vogan:Langlands}.

For $\delta\in H^1(F,G)$, let $G_\delta$ be the pure inner form of $G$ associated with $\delta$. 
Let $\Pi(G_\delta)$ be the set of equivalence classes of irreducible smooth representations of $G_\delta(F)$ and set 
\[
\Pi^{\pure}(G)\ceq\coprod_{\delta\in H^1(F,G)}\Pi(G_\delta).
\]

 We now describe the desiderata of the local Langlands corresponds which is relevant for our purpose. 
\begin{enumerate}
\item\label{desi1} 
There is a finite to one surjective map 
\[
\Pi^{\pure}(G)\to \Phi(G^*).
\]
For each $\phi\in \Phi(G^*)$, let 
\[
\Pi_\phi^{\pure}(G)\subset \Pi^{\pure}(G)
\]
be the pre-image of $\phi$, which is the pure $L$-packet of $\phi$. 

\item\label{desi2} 
Set $A_\phi\ceq \pi_0(Z_{\dualgroup{G}}(\phi))$.
For each Whittaker datum $\mathfrak{w}\in D$, there is a bijection 
\[
J(\mathfrak{w}): \Pi_\phi^{\pure}(G)\to \wh{A_\phi},
\]
which we refer to as a \emph{Whittaker normalization}.
Following \cite{Kaletha}, we denote the inverse of $J(\mathfrak{w})$ by $\iota_{\mathfrak{w}}$.
Each pure L-packet $\Pi^\pure_\phi(G)$ contains at most one $\theta'$-generic representation for each $\mathfrak{w}' = (B',\theta')\in D$. 
If $\Pi^\pure_\phi(G)$ contains a $\theta$-generic representation $\pi$, then $J(\mathfrak{w})(\pi)$ is the trivial representation of $A_\phi$.

\item\label{wittwisting} 
If $\mathfrak{w}'\in D$ is another Whittaker datum, then for any $\phi\in \Phi(G^*)$ and $\rho\in \wh A_{\phi}$, we have 
\[
\iota_{\mathfrak{w}}(\rho)
=
\iota_{\mathfrak{w}'}(\rho\otimes \varrho_{\mathfrak{w},\mathfrak{w'}}^\phi)
\]
where $\varrho_{\mathfrak{w},\mathfrak{w}'}^\phi: A_\phi\to \C^\times$ is the character of $A_\phi$ described above in Equation~\eqref{eqn:twisting}, also in \cite{Kaletha}*{Section 4}. 

\end{enumerate}

\begin{remark}
For classical groups, Desiderata~(\ref{desi1}) and (\ref{desi2}) above were proved in \cites{Arthur: book, Mok:Unitary, KMSW:Unitary}. Part (\ref{wittwisting}) of the above desiderata for tempered parameters of symplectic groups and special orthogonal  was proved in \cite{Kaletha}.
\end{remark}

\subsection{ABV-packets}\label{sec:ABV}

In this section we recall the definition of the ABV-packet attached to a Langlands parameter $\phi$ for $G$ from \cite{CFMMX}, while expanding on the dependence on the Whittaker datum $\mathfrak{w}$ through the local Langlands correspondence; see Definition~\ref{definition:ABVw}.

Fix a Whittaker datum $ \mathfrak{w}$ for $G$. 
In Section~\ref{ssec:LLC}, Desiderata (\ref{desi1}) we saw that the local Langlands correspondence predicts a bijection
  \begin{equation}\label{eq: LLC}
J(\mathfrak{w})=  \iota_{\mathfrak{w}}^{-1}: \Pi_{\phi_i}^{\pure}(G)\to \widehat {A_{\phi_i}}\cong \widehat{A_{C_i}},\end{equation}
  where $C_i=C_{\phi_i}$. 
Define
\[
\Pi^\pure_\lambda(G) \ceq \coprod_{\phi_\lambda =\lambda} \Pi^\pure_{\phi}(G)
\]
where the sum is taken over $\dualgroup{G}$-conjugacy classes of Langlands parameters $\phi$ with infinitesimal parameter $\lambda$.
Using \eqref{eq: LLC} we get a bijection
\begin{equation}\label{llc:map}
\begin{array}{rcl}
\mathcal{P}_{\mathfrak{w}}: \Pi^\pure_\lambda(G) &\to&  \Perv_{H_{\lambda}}(V_\lambda)^{\mathrm{simple}}_{/\mathrm{iso}} 
\\
\pi &\mapsto& \IC(\underline{\rho}_{C_\phi})
\end{array}
\end{equation}
where $(\phi,\rho) = J(\mathfrak{w})(\pi)$ is the enhanced Langlands parameter for $\pi$ (so $\phi$ is a Langlands parameter for $\pi$) and $\rho$ is an irreducible representation of $A_\phi$ with pure inner form  $\delta$, and $\underline{\rho}_{C_\phi}$ is the corresponding $H_\lambda$-equivariant local system on $C_\phi$.

\begin{definition}
For a local Langlands parameter $\phi$ with infinitesimal parameter $\lambda$, and a Whittaker datum $\mathfrak{w}$ as above, define
\begin{equation}\label{definition:ABVw}
\Pi_{\phi,\mathfrak{w}}^{\ABVpure}(G)
:=
\{ \pi \in  \Pi_\lambda^\pure(G) \tq  \NEvs_{C_\phi}(\mathcal{P}_{\mathfrak{w}}(\pi))\ne 0\}.
\end{equation}
This set comes equipped with the function 
\begin{equation}\label{eq: Arthur map}
\begin{array}{rcl}
\NEvs_\phi \circ\mathcal{P}_{\mathfrak{w}} : \Pi_{\phi,\mathfrak{w}}^{\ABVpure}(G) &\to&  \Rep(A^\ABV_{\phi}) \\
\pi &\mapsto&  \NEvs_\phi\left(  \mathcal{P}_{\mathfrak{w}}(\pi) \right).
\end{array}
\end{equation}
After the proof of Proposition~\ref{theorem:CoW}, we will use the notation
\[
\Pi_{\phi}^{\ABVpure}(G) \ceq \Pi_{\phi,\mathfrak{w}}^{\ABVpure}(G),
\]
because that proposition shows that $\Pi_{\phi,\mathfrak{w}}^{\ABVpure}(G)$ is independent of $\mathfrak{w}$.
\end{definition}

\section{Whittaker normalization of ABV-packets}

In this section we show that the ABV-packet $\Pi^{\ABVpure}_{\phi,\mathfrak{w}}(G)$ does not depend on Whittaker normalization, while the function $\Pi^{\ABVpure}_{\phi,\mathfrak{w}}(G) \to \Rep(A^\ABV_\phi)$ does; see Theorem~\ref{theorem:CoW}.

\subsection{Compatibility of ABV-packets with L-packets}

To study the compatibility of ABV-packets with the enhanced local Langlands correspondence, we begin by recalling some basic facts about the so-called component groups that appear in both packets.
From Section~\ref{sec:desiderata}, Desiderata~(\ref{desi2}), recall $A_\phi \ceq \pi_0(Z_{\dualgroup{G}}(\phi)) = \pi_1^{H_\lambda}(C_\phi,x_\phi)$.
From Section~\ref{sec:VC}, recall $A^\ABV_{\phi} \ceq \pi_1^{H_\lambda}(\Lambda^\mathrm{gen}_{C_\phi}, (x_\phi, y))$, for any $(x_\phi,y)\in \Lambda^\mathrm{gen}_{C_\phi}$.
These $H_\lambda$-equivariant fundamental groups are related as follows.
By restriction, the projection $T^*(V_\lambda) \to V_\lambda$ defines the $H_\lambda$-equivariant projection 
\[
\Lambda^\mathrm{gen}_{C_\phi} \to C_\phi.
\]
By pull-back, this defines $\Loc_{H_\lambda}(C_\phi) \to \Loc_{H_\lambda}(\Lambda^\mathrm{gen}_{C_\phi})$. 
Since these spaces are connected, the choice of base points as above determines
\[
\Rep(A_\phi) \to \Rep(A^\ABV_\phi).
\]
The function underlying this functor appears in Theorem~\ref{theorem:compatibility}.

\begin{theorem}[\cite{CFMMX}*{Theorem 7.22 (d)}]\label{theorem:compatibility}
Let $G$ be a connected reductive algebraic group. 
Assume the Desiderata of Section~\ref{sec:desiderata} for $G$.
Let $\phi$ be any Langlands parameter for $G$, without imposing the relevance condition of \cite{Borel:Corvallis}*{Section 8.2 (ii)}.
The following diagram commutes:
\[
\begin{tikzcd}
\Pi^\pure_\phi(G) \arrow[>->]{d}  \arrow{rr}{J(\mathfrak{w})} && \Rep(A_\phi) \arrow{d}\\
\Pi^{\ABVpure}_{\phi, \mathfrak{w}}(G) \arrow{rr}{\NEvs_\phi \circ\mathcal{P}_{\mathfrak{w}}} && \Rep(A_\phi^\ABV)
\end{tikzcd}
\]
\end{theorem}

\begin{proof}
    Suppose $\pi \in \Pi^\pure_\phi(G)$. 
    Then $J(\mathfrak{w})(\pi) = (\phi_\pi, \rho)$ for $\rho \in \Rep(A_\phi)$.
    Now $\mathcal{P}_\mathfrak{w}(\pi) = \IC(\underline{\rho}_{C_{\phi_\pi}})$.
    By \cite{CFMMX}*{Theorem 7.22 (d)}, 
    \[
    \NEvs_{C_{\phi_\pi}} \IC(\underline{\rho}_{C_{\phi_\pi}})  = \left( \underline{\rho}_{C_{\phi_\pi}} \boxtimes \1_{C_{\phi_\pi}^*} \right)\vert_{\Lambda_{C_{\phi_\pi}^\mathrm{gen}}}.
    \]
    To finish the proof observe that the map $\Rep(A_\phi) \to \Rep(A^\ABV_\phi)$ is induced from pull back along the $H_\lambda$-equivariant projection
    \[
    \Lambda^\mathrm{gen}_{C_{\phi_\pi}} \to C_{\phi_\pi},
    \]
 and that the pullback of $\underline{\rho}_{C_{\phi_\pi}}$ by this projection is $\left( \underline{\rho}_{C_{\phi_\pi}} \boxtimes \1_{C_{\phi_\pi}^*} \right)\vert_{\Lambda_{C_{\phi_\pi}^\mathrm{gen}}}$.
 This completes the proof of Theorem~\ref{theorem:compatibility}.
\end{proof}

\subsection{Change of Whittaker normalization of ABV-packets}\label{sec:ABVCoW}

It is clear that the definition of the ABV-packet $\Pi_{\phi,\mathfrak{w}}^{\ABVpure}(G)$ depends on the local Langlands correspondence. The next result shows that this set does not depend on the Whittaker datum $\mathfrak{w}$, while the function $\NEvs_\phi  \circ\mathcal{P}_{\mathfrak{w}}  : \Pi_{\phi,\mathfrak{w}}^{\ABVpure}(G) \to \Rep(A^\ABV_{\phi})$ does, and shows exactly how.

First, however, recall that, for every pair of Whittaker datum $\mathfrak{w}, \mathfrak{w}'$, the following diagram commutes, according to the local Langlands correspondence as recalled in Section~\ref{sec:desiderata}.
\begin{equation}\label{eqn:desiderata}
\begin{tikzcd}
\Pi_{\phi}^{\pure}(G) \arrow[equal]{d} \arrow{rr}{J(\mathfrak{w})} &&  \widehat{A_{\phi}} \arrow{d}{\, -\, \otimes \varrho_{\mathfrak{w},\mathfrak{w}'}^{\phi}}\\
\Pi_{\phi}^{\pure}(G) \arrow{rr}{J(\mathfrak{w}')} && \widehat{A_{\phi}} 
\end{tikzcd}
\end{equation}
This motivates our next result, Theorem~\ref{theorem:CoW}, which expands on the compatibility of ABV-packets with the local Langlands correspondence, as expressed in Theorem~\ref{theorem:compatibility}.
  
\begin{theorem}\label{theorem:CoW}
For every Langlands parameter $\phi$, the ABV-packet $\Pi_{\phi,\mathfrak{w}}^{\ABVpure}(G)$ is independent of the choice of Whittaker datum $\mathfrak{w}$, while the function $\NEvs_\phi \circ\mathcal{P}_{\mathfrak{w}} : \Pi_{\phi,\mathfrak{w}}^{\ABVpure}(G)\to \Rep(A^\ABV_{\phi})$ does
depend on $\mathfrak{w}$, as follows: for every Whittaker datum $\mathfrak{w}'$, the following diagram commutes
\[
\begin{tikzcd}
\Pi_{\phi,\mathfrak{w}}^{\ABVpure}(G) \arrow[equal]{d} \arrow{rr}{\NEvs_\phi \circ\mathcal{P}_{\mathfrak{w}}} &&  \Rep(A^\ABV_{\phi}) \arrow{d}{\, -\, \otimes \varrho_{\mathfrak{w},\mathfrak{w}'}^{\ABV,\phi}}\\
\Pi_{\phi,\mathfrak{w}'}^{\ABVpure}(G) \arrow{rr}{\NEvs_\phi \circ\mathcal{P}_{\mathfrak{w}'}} && \Rep(A^\ABV_{\phi}) 
\end{tikzcd}
\]
where the character $\varrho_{\mathfrak{w},\mathfrak{w}'}^{\ABV,\phi}$ of $A^\ABV_\phi$ is the image of $\varrho_{\mathfrak{w},\mathfrak{w}'}^{\lambda}$ under $\Rep(A_\phi) \to \Rep(A^\ABV_\phi)$ appearing in Theorem~\ref{theorem:compatibility}. 
\end{theorem}

\begin{proof}
Fix Whittaker data $\mathfrak{w}$ and $\mathfrak{w}'$.
Fix an infinitesimal parameter $\lambda$.
Recall the character $\varrho_{\mathfrak{w},\mathfrak{w}'}^\lambda$ of $A_\lambda = \pi_0(H_\lambda)$ defined in Section~\ref{ssec:LLC}, Equation~\eqref{eqn:twistinglambda}.
We use \cite{CFMMX}*{Proposition~4.3},
to view $\varrho_{\mathfrak{w},\mathfrak{w}'}^\lambda$ as a local system on $V_\lambda$, using
\[
\begin{tikzcd}
\Rep(\pi_0(H_\lambda)) \arrow{rrr}
 &&& \Loc_{H_\lambda}(V_\lambda) \arrow[shift left]{rrr}{\mathrm{forget}} &&&  \Loc_{H_\lambda^0}(V_\lambda) 
\end{tikzcd}
\]
in which we use the notation 
\[
\underline{\varrho_{\mathfrak{w},\mathfrak{w}'}}_{V_\lambda} \in \Loc_{H_\lambda}(V_\lambda)
\]
 for the image of $\varrho_{\mathfrak{w},\mathfrak{w}'}^\lambda$ under 
$
\Rep(\pi_0(H_\lambda)) \to \Loc_{H_\lambda}(V_\lambda).
$
Observe that the image of this local system is the constant sheaf under the functor that forgets to $H_\lambda^\circ$-equivariant sheaves.
Likewise, for any Langlands parameter, recall the character $\varrho_{\mathfrak{w},\mathfrak{w}'}^\phi$ of $A_\phi$ from Equation~\eqref{eqn:twisting} and 
use the equivalence 
\[
\Rep(A_\phi) \to \Loc_{H_\lambda}(C_\phi)
\]
to view $\varrho_{\mathfrak{w},\mathfrak{w}'}^\phi$ as a local system on $C_\phi$, denoted now by
\[
\underline{\varrho_{\mathfrak{w},\mathfrak{w}'}}_{C_\phi} \in \Loc_{H_\lambda}(V_\lambda).
\]
By Lemma~\ref{lemma:restriction}, if $\phi$ has infinitesimal parameter $\lambda$, then
\[
\left(\underline{\varrho_{\mathfrak{w},\mathfrak{w}'}}_{V_\lambda}\right)\vert_{C_\phi} = \underline{\varrho_{\mathfrak{w},\mathfrak{w}'}}_{C_\phi}.
\]

Now suppose $\pi \in \Pi^\pure_\lambda(G)$.
Then $J(\mathfrak{w})(\pi) = (\phi_\pi,\rho)$, for a unique $\rho\in \widehat{A_\phi}$; 
equivalently, $\iota_{\mathfrak{w}}(\rho) = \pi$.
Then 
\[
\mathcal{P}_{\mathfrak{w}} (\pi) =  \IC(\underline{\rho}_{C_{\phi_\pi}}).
\]
From Section~\ref{sec:desiderata}, $\iota_{\mathfrak{w}}(\rho) = \iota_{\mathfrak{w}'}(\rho\otimes \varrho^{\phi_\pi}_{\mathfrak{w},\mathfrak{w}'})$, 
so $J(\mathfrak{w}')(\pi) = (\phi_\pi,\rho\otimes \varrho^{\phi_\pi}_{\mathfrak{w},\mathfrak{w}'})$. 
Thus,
\[
\mathcal{P}_{\mathfrak{w}'} (\pi) =  \IC(\underline{\rho\otimes \varrho_{\mathfrak{w},\mathfrak{w}'}}_{C_{\phi_\pi}}).
\]
By Lemma~\ref{lemma:twistedIC},
\[
\IC(\underline{\rho\otimes \varrho_{\mathfrak{w},\mathfrak{w}'}}_{C_{\phi_\pi}})
=
\IC(\underline{\rho}_{C_{\phi_\pi}}) \otimes \underline{\varrho_{\mathfrak{w},\mathfrak{w}'}}_{V_\lambda}.
\]

Now, in order to calculate $\NEvs_{C_\phi} \mathcal{P}_{\mathfrak{w}'} (\pi)$ we define
\begin{equation}\label{eqn:conormal}
\underline{\varrho_{\mathfrak{w},\mathfrak{w}'}}_{\Lambda^\mathrm{gen}_{C}}
\ceq
\left( \underline{\varrho_{\mathfrak{w},\mathfrak{w}'}}_{V_\lambda} \boxtimes \1_{V_\lambda^*} \right)\vert_{\Lambda_{C_{\phi_\pi}^\mathrm{gen}}}.
\end{equation}
Then
\[
\begin{array}{rclr}
\NEvs_{C_\phi} \mathcal{P}_{\mathfrak{w}'} (\pi)
&=& \NEvs_{C_\phi} \IC(\underline{\rho\otimes \varrho_{\mathfrak{w},\mathfrak{w}'}}_{C_{\phi_\pi}})& \text{by Definition~\ref{definition:ABVw}}\\
&=& \NEvs_{C_\phi} \left( \IC(\underline{\rho}_{C_{\phi_\pi}}) \otimes \underline{\varrho_{\mathfrak{w},\mathfrak{w}'}}_{V_\lambda}\right)& \text{by Lemma~\ref{lemma:twistedIC}}\\
&=& \left( \NEvs_{C_\phi} \IC(\underline{\rho}_{C_{\phi_\pi}}) \right) \otimes \underline{\varrho_{\mathfrak{w},\mathfrak{w}'}}_{\Lambda^\mathrm{gen}_{C_{\phi_\pi}}}& \text{by Lemma~\ref{lemma:twistedNEvs}} \\
&=& \left( \NEvs_{C_\phi} \mathcal{P}_{\mathfrak{w}}(\pi) \right) \otimes \underline{\varrho_{\mathfrak{w},\mathfrak{w}'}}_{\Lambda^\mathrm{gen}_{C_{\phi}}}& \text{by Definition~\ref{definition:ABVw}}. \\
\end{array}
\]
Therefore,
\[
\begin{array}{rcl}
\NEvs_{\phi} \circ\mathcal{P}_{\mathfrak{w}'} (\pi)
&=& \left( \NEvs_{\phi} \circ\mathcal{P}_{\mathfrak{w}}(\pi) \right) \otimes \varrho_{\mathfrak{w},\mathfrak{w}'}^{\ABV,\phi}, \\
\end{array}
\]
This concludes the proof of Proposition~\ref{theorem:CoW}
\end{proof}

\begin{lemma}\label{lemma:restriction}
\[
\left(\underline{\varrho_{\mathfrak{w},\mathfrak{w}'}}_{V_\lambda}\right)\vert_{C} = \underline{\varrho_{\mathfrak{w},\mathfrak{w}'}}_{C}.
\]
\end{lemma}

\begin{proof}
This follows from $Z_{\dualgroup{G}}(\phi) \hookrightarrow Z_{\dualgroup{G}}(\lambda)$, for any Langlands parameter $\phi$ with infinitesimal parameter $\lambda$, which induces the map $\iota_{\lambda \leq \phi} : A_\phi \to A_\lambda$ appearing at the end of Section~\ref{sec:CoWh}.
\end{proof}

The following lemmas describe how changing Whittaker normalization acts on the fibers of the forgetful functor appearing in 
\cite{CFMMX}*{Proposition~4.3}.
\[
\begin{tikzcd}
\Rep(\pi_0(H_\lambda)) \arrow{rrr}
 &&& \Loc_{H_\lambda}(V_\lambda) \arrow[shift left]{rrr}{\mathrm{forget}} &&&  \Loc_{H_\lambda^0}(V_\lambda) 
\end{tikzcd}
\]

\begin{lemma}\label{lemma:twistedIC}
If $\mathcal{L}_C$ is a simple $H_\lambda$-equivariant local system on an $H_\lambda$-orbit $C$ in $V_\lambda$ then
\[
\IC(\mathcal{L}_C\otimes \underline{\varrho_{\mathfrak{w},\mathfrak{w}'}}_{C} ) = \IC(\mathcal{L}_C)\otimes \underline{\varrho_{\mathfrak{w},\mathfrak{w}'}}_{V_\lambda}
\]
\end{lemma}

\begin{proof}
That $\IC(\mathcal{L}_C)\otimes \underline{\varrho_{\mathfrak{w},\mathfrak{w}'}}_{V_\lambda}$ is perverse can be seen by applying the forgetful functor $\Perv_{H_\lambda}(V_\lambda) \to \Perv(V_\lambda)$ to this complex and observing that the image is exactly $\IC(\mathcal{L}_C)$, since the image of $\underline{\varrho_{\mathfrak{w},\mathfrak{w}'}}_{V_\lambda}$ is the constant sheaf on $V_\lambda$.
By Lemma~\ref{lemma:restriction}, the restriction of $\IC(\mathcal{L}_C)\otimes \underline{\varrho_{\mathfrak{w},\mathfrak{w}'}}_{V_\lambda}$
to $C$ is $\mathcal{L}_C \otimes\underline{\varrho_{\mathfrak{w},\mathfrak{w}'}}_{C}$.
The same is true of the left hand side. 
The lemma now follows from the classification of simple objects in $\Perv_{H_\lambda}(V_\lambda)$.
\end{proof}

\begin{lemma}\label{lemma:twistedNEvs}
Suppose $\mathcal{F}\in \Perv_{H_\lambda}(V_\lambda)$ is simple. 
Then
\[
\NEvs_{C} ( \mathcal{F}\otimes \underline{\varrho_{\mathfrak{w},\mathfrak{w}'}}_{V_\lambda} ) = \NEvs_{C} ( \mathcal{F} )\otimes \underline{\varrho_{\mathfrak{w},\mathfrak{w}'}}_{\Lambda^\mathrm{gen}_{C}},
\]
for every $H_\lambda$-orbit $C$ in $V_\lambda$.
Equivalently, for any Langlands parameter $\phi$ for $G$,
\[
\NEvs_{\phi} \left(\mathcal{F}\otimes \underline{\varrho_{\mathfrak{w},\mathfrak{w}'}}_{V_\lambda}\right)
= 
\left(\NEvs_{\phi} \mathcal{F}\right) \otimes \varrho_{\mathfrak{w},\mathfrak{w}'}^{\ABV,\phi},
\]
where $\lambda$ is the infinitesimal parameter for $\phi$.
\end{lemma}

\begin{proof}
Recall that $\varrho_{\mathfrak{w},\mathfrak{w}'}^{\ABV,\phi}$ denotes the image of $\varrho_{\mathfrak{w},\mathfrak{w}'}^{\lambda}$ under $\Rep(A_\phi) \to \Rep(A^\ABV_\phi)$, induced by the $H_\lambda$-equivariant projection $\Lambda^\mathrm{gen}_{C_\phi} \to C_\phi$.
\[
\underline{\varrho_{\mathfrak{w},\mathfrak{w}'}}_{\Lambda^\mathrm{gen}_{C_\phi}} 
= 
\left( \underline{\varrho_{\mathfrak{w},\mathfrak{w}'}} \boxtimes \1_{C_{\phi_\pi}^*} \right)\vert_{\Lambda_{C_{\phi_\pi}^\mathrm{gen}}}
\]
Using this, and arguing as in the proof of \cite{CFMMX}*{Proposition~7.13}, gives
\[
\begin{array}{rcl}
\NEvs_{C} ( \mathcal{F}\otimes \underline{\varrho_{\mathfrak{w},\mathfrak{w}'}}_{V_\lambda} )     
&=& 
\mathcal{T}_{C}^\vee\otimes
\left(\RPhi_{f_{C}}((\mathcal{F}\otimes \underline{\varrho_{\mathfrak{w},\mathfrak{w}'}}_{V_\lambda})\boxtimes \1_{C^*})\right)\vert_{\Lambda^\mathrm{gen}_{C}} \\
&=& 
\mathcal{T}_{C}^\vee\otimes
\left(\RPhi_{f_{C}}\left((\mathcal{F}\boxtimes \1_{C^*}) \otimes ( \underline{\varrho_{\mathfrak{w},\mathfrak{w}'}}_{V_\lambda}\boxtimes \1_{C^*})\right)\right)\vert_{\Lambda^\mathrm{gen}_{C}} \\
&=& \mathcal{T}_{C}^\vee\otimes
\left(\RPhi_{f_{C}}(\mathcal{F}\boxtimes \1_{C^*})\right)\vert_{\Lambda^\mathrm{gen}_{C}}
\otimes 
\left( \underline{\varrho_{\mathfrak{w},\mathfrak{w}'}}_{V_\lambda}\boxtimes \1_{C^*}\right)\vert_{\Lambda^\mathrm{gen}_{C}} \\
&=&
(\NEvs_{C} \mathcal{F})
\otimes 
\underline{\varrho_{\mathfrak{w},\mathfrak{w}'}}_{\Lambda^\mathrm{gen}_{C}}.
\end{array}
\]
The first and fourth equalities are by definition of $\NEvs_C$ \cite{CFMMX}*{Section 7.10}.
The second equality is elementary. 
The third equality uses \cite{CFMMX}*{Lemma 7.14} as it is used in the proof of \cite{CFMMX}*{Proposition~7.13}. 
Note that $\underline{\varrho_{\mathfrak{w},\mathfrak{w}'}}_{V_\lambda}\boxtimes \1_{C^*}$ forgets to  the constant sheaf $\1_{V_\lambda \times C^*}$ on $V_\lambda \times C^*$.
\end{proof}

\begin{remark}\label{remark:disconnectedH}
Unless the group $H_\lambda$ is disconnected, it is unusual for there to be more than one irreducible $H_\lambda$-equivariant \'etale local system on all of $V_\lambda$.
When the group $H_\lambda$ is disconnected each irreducible representation $\varrho$ of $ H_\lambda/H_\lambda^o$ gives rise to a local system $\underline{\varrho}_{V_\lambda}$ on all of $V_\lambda$.
This applies in particular if we assume Desiderata~(\ref{wittwisting}) of Section~\ref{ssec:LLC}.
We believe that if there is more than one Whittaker normalization for an $L$-packet $\Pi_\phi(G)$, the corresponding group $H_\lambda$ is disconnected, where $\lambda$ is the infinitesimal parameter of $\phi$.
\end{remark}

\section{Open parameters}

In this section we prove the main result of this paper: if $\phi$ is an open parameter for $G$ then the ABV-packet $\Pi_\phi^{\ABVpure}(G)$ coincides with the pure L-packet $\Pi^\pure_\phi(G)$ and the function 
\[
J(\mathfrak{w}) : \Pi^\pure_\phi(G) \to \widehat{A_\phi}
\]
from the local Langlands correspondence coincides with the one given by vanishing cycles following \cite{CFMMX}.

\subsection{Open parameters}

\begin{definition}[\cite{CFZ:unipotent}*{\S 0.6}]\label{def:open}
    A Langlands parameter $\phi : W'_F \to \Lgroup{G}$ is said to be \emph{open} (resp. \emph{closed}) if the corresponding point $x_\phi \in V_\lambda$ (see Subsection \ref{variety}) lies in the open (resp. closed) $H_\lambda$-orbit in $V_\lambda$, where $\lambda$ is the infinitesimal parameter of $\phi$. 
\end{definition}

We will see that tempered parameters are precisely those open parameters that are of Arthur type; see Proposition~\ref{prop: temper=arthur+open}.
We also show that $\phi$ is open if and only if the adjoint $L$-function $L(s,\phi,\Ad)$ is holomorphic at $s=1$; see Proposition~\ref{prop:openL}.
Taken together, these facts tell us that openness is an appropriate generalization of temperedness for Langlands parameters from the case of those of Arthur type.

\subsection{Open versus tempered parameters}

Proposition~\ref{prop: temper=arthur+open} shows the relationship between open parameters and tempered parameters.
    We take the definition of tempered representations as given in \cite{wald}*{Prop III.2.2.}.
    A Langlands parameter $\phi$ is said to be tempered if and only if its restriction to $W_F$ is bounded in $\dualgroup{G}$.
    As in Borel's desiderata, see Section \ref{ssec:LLC}, it is expected that a Langlands parameter $\phi : W'_F \to \Lgroup{G}$ is \emph{tempered} if and only if its $L$-packet $\Pi_\phi(G)$ contains a tempered representation; in this case we say the $L$-packet itself is tempered.

    Recall that a Langlands parameter $\phi : W'_F \to \Lgroup{G}$ is said to be of \emph{Arthur type} if there is an Arthur parameter $\psi : W''_F \to \Lgroup{G}$, where $W''_F\ceq W_F \times \SL_2(\CC)\times \SL_2(\CC)$,  such that 
    \[
    \phi(w,x) = \psi(w,x,d_w),
    \]
    where $d_w\ceq \diag(\abs{w}_F^{1/2},\abs{w}_F^{-1/2})$.

\begin{remark}
    An irreducible representation $\pi$ is said to be of Arthur type if it appears in an $A$-packet. There are many examples of irreducible representations of Arthur type for which the corresponding Langlands parameter is not of Arthur type.
\end{remark}

\begin{proposition}\label{prop: temper=arthur+open}
A Langlands parameter is tempered if and only if it is open and of Arthur type.
\end{proposition}

\begin{proof}
Suppose $\phi : W'_F \to \Lgroup{G}$ is tempered. 
Define $\psi: W''_F \to \Lgroup{G}$ by $\psi(w,x,y) = \phi(w,x)$.
Then the image of the restriction of $\psi$ to $W_F$ is the image of the restriction of $\phi$ to $W_F$, which is bounded since $\phi$ is tempered. 
Consequently, $\psi$ is an Arthur parameter. Since $\phi(w,x) = \psi(w,x,d_w)$, it follows that $\phi$ is of Arthur type.
To see that $\phi$ is open, note that $x_\phi= x_\psi \in V_\lambda$, where $x_\phi$ and $x_\psi$ are defined by the property $\exp x_\phi = \phi(1,e)$ \cite{CFMMX}*{Proposition 4.2.2} and $\exp x_\psi = \psi(1,e,1)$ \cite{CFMMX}*{Section 6.6}. 
By \cite{CFMMX}*{Proposition 6.6.1}, $(x_\psi,y_\psi) \in \Lambda_\lambda^\mathrm{sreg}$, where $y_\psi\in V^*_\lambda$ is defined by the property $\exp y_\psi = \psi(1,1,f)$. 
Since $\psi(1,1,f) = \phi(1,1) =1$, it follows that $y_\psi = 0$ and therefore that $C_\phi^* = \{ 0 \}$. 
By Proposition ~\ref{lemma:open} it follows that $C_\phi$ is the open orbit in $V_\lambda$. This completes the proof that $\phi$ is open if $\phi$ is tempered.

Now suppose $\phi$ is open and of Arthur type. 
Then $\phi(w,x) = \psi(w,x,d_w)$ for a unique Arthur parameter $\psi$; recall that, by definition, the image of the restriction of $\psi$ to $W_F$ is bounded in $\dualgroup{G}$.
Note that $x_\phi = x_\psi$ as above and also that $(x_\psi,y_\psi)\in \Lambda_\lambda^\mathrm{sreg}$, as above.
Again by Proposition ~\ref{lemma:open}, since $\phi$ is open it follows that $C_\phi^* = \{ 0 \}$ so $y_\psi =0$. 
If follows that $\psi(1,1,f) =1$. Now consider $\sigma \ceq \psi^\circ\vert_{\SL_2^{\hskip-1pt\mathrm{Art}}(\CC)} : \SL_2(\CC) \to \dualgroup{G}$ and note that $\sigma(f) =1$. Arguing as in \cite{GR}*{Section 2}, it follows from the Jacobson-Morozov theorem that $\sigma$ is determined uniquely by the $\SL_2$-triple $\sigma(e)$, $\sigma(h)$ and $\sigma(f)$. Since $\sigma(f)=1$, this triple is trivial, so $\sigma$ is trivial. 
It follows that $\phi(w,x) = \psi(w,x,y)$.
Since $\psi(w,1) = \psi(w,1,d_w)$ and since $\psi(w,1,d_w) = \psi(w,1,1)$ it follows that the image of the restriction of $\phi_\psi$ to $W_F$ in $\dualgroup{G}$ is equal to the image of the restriction of $\psi$ to $W_F$ in $\dualgroup{G}$, which is bounded. This concludes the proof that if $\psi$ is open and of Arthur type then $\phi$ is tempered.
\end{proof}

    It follows from Proposition~\ref{prop: temper=arthur+open} that tempered parameters are open. 
    We give an alternate proof of this fact, as well as the equivalent statement for discrete Langlands parameters in Appendix~\ref{sec:discrete}. This alternate proof is a straightforward consequence of some connection established in \cite{heiermannorbit} between discrete Langlands parameters and distinguished nilpotent elements, and basic theorems in the theory of distinguished nilpotent orbits.

\begin{remark} 
Langlands parameters that are open need not be tempered, nor of Arthur type.
For example, consider the Langlands parameter $\phi_s : W'_F \to \Lgroup{G}$ for $G = \GL_n$ defined by 
\[
\phi_s(w,x) = \abs{w}_F^{s} \otimes\Sym^{n-1}(x)  
\]
where $s$ is an arbitrary complex number. 
Then the infinitesimal parameter for $\phi_s$ is 
\[
\lambda_s(w)=\diag(\abs{w}_F^{s+\frac{n-1}{2}}, \abs{w}_F^{s+\frac{n-3}{2}}, \ldots, \abs{w}_F^{s-\frac{n-3}{2}},  \abs{w}_F^{s-\frac{n-1}{2}}) \in \dualgroup{T}.
\]
The variety $V_{\lambda_s}$ is independent of $s$; specifically, for every $s$, the variety $V_{\lambda_s}$ for this infinitesimal parameter is $\mathbb{A}^{n-1}$, with $2^{n-1}$ orbits under the action of $H_{\lambda_s} = \dualgroup{T}$.
(This variety and the corresponding representations are studied in \cite{CS:Steinberg}.)
Moreover, for every $s$, the parameter $\phi_s$ lies in the open orbit, hence is an open parameter. 
However, for $s\ne 0$ the Langlands parameter $\phi_s$ is open but not tempered, nor of Arthur type; meanwhile, for $s=0$ the Langlands parameter $\phi_s$ is open and tempered (and thus of Arthur type).
\end{remark}

\begin{remark}
It may be tempting to think, based on the remark above, that "open" is equivalent to "essentially tempered", which is to say, the twist of a parameter for a tempered representation of $G(F)$ by the parameter for an unramified character of $G(F)$, but that is not true, as the example of the Langlands parameter for $\GL_{n_1+n_2}(F)$
\[
\left(\abs{w}_F^{s_1} \otimes\Sym^{n_1-1}(x)\right) \oplus \left( \abs{w}_F^{s_2} \otimes\Sym^{n_2-1}(x) \right)
\]
shows, which for $s_2-s_1 \ne \pm 1$ is open and not essentially tempered.
\end{remark}

\subsection{Open parameters and L-functions}

We now give an alternate perspective on open parameters.
Let $\phi: W_F\times \SL_2(\C)\to \Lgroup{G}$ be a local Langlands parameter and let $\Ad:\Lgroup{G}\ra \Aut({\hat \fg})$ be the adjoint representation of $\Lgroup{G}$ on the Lie algebra of the dual group of $G$.
The corresponding local $L$-factor is 
\[
L(s,\phi,\Ad)=\det(I-q^{-s}\Ad(\lambda(\Fr))|_{\wh \fg^{I_F}_N})^{-1},
\]
where $\lambda:=\lambda_\phi$ (see \eqref{eq-infinitesimal}) is the infinitesimal parameter of $\phi$, $N=d\phi\bspm 0&1\\ 0 &0 \espm$ and $\wh\fg_{N}=\ker(N)$ and $\wh\fg_N^{I_F}$ is the $I_F$-invariant subspace of $\wh\fg_N$.  See \cite{GR} for more details about the definition of local $L$-factors. 
We remark that the Frobenius in \cite{GR} is the inverse of the Frobenius in \cite{CFMMX}.


\begin{proposition}\label{prop:openL}
For any Langlands parameter $\phi$, the $L$-function $L(s,\phi,\Ad)$ is regular at $s=1$ if and only if $\phi$ is open. 
\end{proposition}

\begin{proof}
Write $C=C_\phi$ in the following. Note that $N\in C$ and $L(s,\phi,\Ad)$ has a pole at $s=1$ if and only if there exists a nonzero element $y\in \wh \fg_N^{I_F}$ such that $\Ad(\lambda(\Fr)) y =q y$. 
If follows that $y\in V_\lambda^*$ and $[N,y]=0$, where we take the Lie bracket in $\wh \fg$.
Thus, $(N,y)$ lies in the conormal bundle $\Lambda_{C}$. 
Equivalently, $L(s,\phi,\Ad)$ is regular at $s=1$ if and only if $\{y\in V_{\lambda}^*\tq (N,y)\in \Lambda_C\}=\{0\}.$

We first assume that $N=0$, i.e., $\phi(w,x)=\lambda(w)$. Then $L(s,\phi,\Ad)$ is regular at $s=1$ if and only if $V_\lambda^*=\{0\}$. The latter condition is equivalent to $V_\lambda=\{0\}$, in which case $\phi=\lambda$ is trivially open in $V_\lambda$.

Next, we assume that $N\ne 0$. Since $H_\lambda.N=C$ and $[~,~]$ is invariant under the $H_\lambda$-action, it follows that $L(s,\phi,\Ad)$ is regular at $s=1$ if and only if the projection $\mathrm{pr}_2: \Lambda_C\ra V_{\lambda}^*$ is zero, which implies $C^*=\{0\}$, hence $C$ is open. 
%
\end{proof}


\section{A genericity conjecture for ABV-packets}

\subsection{Generic versus open}\label{sec:genL}

    Recall that an irreducible representation $\pi$ is said to be generic if it has a Whittaker model.
    A Langlands parameter $\phi$ is said to be generic if its L-packet $\Pi_\phi(G)$ contains a generic representation; in this case we say the $L$-packet itself is generic.
    It is possible for an L-packet to have more than one generic representation; a simple example is the unique L-packet for $\SL_2(F)$, as in \cite{CFMMX}*{Chapter 11}. 
    It is possible for an L-packet to have exactly one one generic representation while all others are not; an example is the unique packet for $G_2(F)$ with component group $S_3$, as in \cite{CFZ:cubic}*{Section 1.1}.
Of course, most L-packets have no generic representations.

Recall the following conjecture of Gross-Prasad and Rallis from \cite{GP}*{Conjecture 2.6}:
{An $L$-packet $\Pi_\phi(G(F))$ is generic if and only if $L(s,\phi,\Ad)$ is regular at $s=1$.}
This conjecture has been verified in many cases in the literature; see \cites{Jiang-Soudry, Liu, JL} for some examples. 
The most general case of \cite{GP}*{Conjecture 2.6} proved in the literature is \cite{GI}*{Theorem B.2}:
If $G$ is a classical group (including $GL_n$), then \cite{GP}*{Conjecture 2.6} is true.
In fact, Gan-Ichino \cite{GI} proved \cite{GP}*{Conjecture 2.6} for any reductive group under certain hypothesis on local Langlands correspondence, see \cite{GI}*{Section B.2},  which are known to be true for classical groups (including the general linear groups).

As a consequence of Proposition~\ref{prop:openL}, \cite{GP}*{Conjecture 2.6}  is equivalent to the following conjectural geometric characterization of generic $L$-packets.

\begin{conjecture}\label{conj: GP2}
A pure $L$-packet $\Pi^\pure_\phi(G)$ is generic if and only if $\phi$ is an open parameter for $G$.
\end{conjecture}

For the exceptional group $G_2$, one can also check directly that \cite{GP}*{Conjecture 2.6} and Conjecture~\ref{conj: GP2} hold for unramified parameters $\phi$ using \cites{CFZ:cubic, CFZ:unipotent} and for general $\phi$ using the  local Langlands correspondence for $G_2$ in \cites{Aubert-Xu, Gan-Savin}.

\begin{corollary}\label{corollary: GI}
If $G$ is a classical group (including $GL_n$), then  Conjecture~\ref{conj: GP2} is true: $\phi$ is generic if and only if $\phi$ is open.
\end{corollary}

\begin{proof}
By \cite{GI}*{Theorem B.2}, the conjecture of Gross-Prasad and Rallis, \cite{GP}*{Conjecture 2.6} is true for classical groups. This means that a representation of a classical group is generic if and only if $L(s,\phi, \Ad)$ is regular at $s=1$. Proposition~\ref{prop:openL} says that $\phi$ is open if and only if $L(s, \phi, \Ad)$ is regular at $s=1$.
\end{proof}

\subsection{Genericity conjecture for ABV-packets}

\begin{conjecture}[Genericity Conjecture for ABV-packets]\label{conj: genericabv}
For any connected reductive algebraic group over $F$, an ABV-packet $\Pi_\phi^{\ABV\pure}(G)$ contains a generic representation if and only if $\phi$ is open, where $\phi$ is a Langlands parameter for $G$.
\end{conjecture}

Conjecture \ref{conj: genericabv} is a generalization of \cite{GP}*{Conjecture 2.6} in view of Proposition~\ref{prop:openL}. 
The new content of Conjecture \ref{conj: genericabv} compared to \cite{GP}*{Conjecture 2.6} is that $\Pi_\phi^{\ABV\pure}(G)$ does not contain any generic representations if $C_\phi$ is not open in $V_{\lambda_\phi}$. 
Conjecture~\ref{conj: genericabv} is a generalization of \cite{LLS}*{Conjecture 1.2} if we admit Vogan's conjecture on A-packets \cite{CFMMX}*{Conjecture 1, Section 8.3}. 
Recall that $\Pi^\pure_\phi(G)\subset \Pi_\phi^{\ABV\pure}(G)$ for general Langlands parameter $\phi$ by Theorem~\ref{theorem:compatibility} and $\Pi^\pure_\phi(G)=\Pi_{\phi}^{\ABVpure}(G)$ if $\phi$ is open, by Theorem~\ref{theorem:mainopen}. 

On the other hand, the following theorem shows that \cite{GP}*{Conjecture 2.6} is indeed equivalent to Conjecture \ref{conj: genericabv} for quasi-split reductive groups.

\begin{theorem}\label{theorem:maingeneric}
Let $G$ be a quasi-split reductive group over $F$ and $\phi$ be a local Langlands parameter of $G$. 
We assume that the local Langlands correspondence holds for $G$ which also satisfies the desiderata of Section~\ref{ssec:LLC}. 
If \cite{GP}*{Conjecture 2.6} is true then $\Pi_\phi^{\ABV\pure}(G)$ contains a generic representation if and only $\phi$ is open. 
\end{theorem}

\begin{proof}
Let $\lambda$ be the infinitesimal parameter of $\phi$ and let $C^o$ be the open orbit of $V_\lambda$ and  let $\phi^o$ be the open Langlands parameter. 
It suffices to show that if $\phi$ is not open, then $\Pi_\phi^{\ABVpure}(G)$ cannot contain a generic representation. 
By Proposition~\ref{prop:openL}, \cite{GP}*{Conjecture 2.6} and Desiderata (\ref{desi2}) of Section \ref{ssec:LLC}, if $\pi$ is a generic representation in $\Pi^\pure_\lambda(G)$, there is a Whittaker datum $\mathfrak{w}$ such that $J(\mathfrak{w})(\pi)=(C^o,\1)$ under the local Langlands correspondence. 
By Lemma~\ref{cor: generic appear in every standard module}, the only ABV-packet which contains $\pi$ is $\Pi^{\ABVpure}_{\phi^o}(G)$.
\end{proof}

\begin{lemma}\label{cor: generic appear in every standard module}
Assume Desiderata (\ref{wittwisting}) of Section \ref{ssec:LLC}. 
Let $C^o$ be the open orbit of $V_\lambda$.
If there exists a Whitaker normalization for which $J(\mathfrak{w})(\pi) = (C^o,\1)$ then in any normalization:
the only $L$-packet containing $\pi$ is the one for $C^o$ and the only ABV-packet containing $\pi$ is the one for $C^o$.
\end{lemma}

\begin{proof}
The first claim follows from Section~\ref{ssec:LLC} because under a different normalization we would have $J(\mathfrak{w}')(\pi) = (C^o, \varrho_{\mathfrak{w},\mathfrak{w}'})$ so that the associated $L$-packet remains that for $C^o$.
The second claim follows from Theorem~\ref{theorem:CoW} and the observation that the only ABV-packet containing $\iota_{\mathfrak{w}}(C^o,\1)$ where $C^o$ is the open orbit, is the ABV-packet for the open orbit.
Indeed, because $V_\lambda$ is smooth and $IC({\1}_{C^o}) = \1_{V_\lambda}[d]$ we have that $\NEvs_{C} (  \IC({\1}_{C^o}))  \neq 0$ if and only if $\Evs_{C} (  \IC({\1}_{C^o}))  \neq 0$ if and only if $\RPhi_{f_{C}}\left(\1_V[d]\boxtimes \1_{C^*}\right)\vert_{\Lambda^\mathrm{gen}_{C}}\neq 0$ if and only if $C=C^o$.
\end{proof}

We remind the reader that a generic representation appearing in a pure L-packet $\Pi^\pure_\phi(G)$ necessarily lies in the L-packet for the quasi-split group $G$.

\begin{remark}
Assuming Desiderata~(\ref{desi2}) of Section~\ref{ssec:LLC}, for classical groups if $C^o$ is open then $\iota_{\mathfrak{w}}(C^o,\1)$ is generic; see Corollary~\ref{corollary: GI} and Proposition~\ref{prop:openL}.
\end{remark}

\begin{corollary}\label{corollary: classical}
Let $G$ be a quasi-split classical group over $F$ or a pure inner form of such.
Then $\Pi^{\ABVpure}_\phi(G)$ contains a generic representation if and only if $L(s,\phi,\Ad)$ is regular at $s=1$.
\end{corollary}

\begin{proof}
This follows from \cite{GI}*{Theorem B.2}, Proposition~\ref{prop:openL} and Theorem~\ref{theorem:maingeneric} directly. 
\end{proof}

\subsection{Relation to a weak form of Vogan's conjecture on A-packets}

Note that Corollary~\ref{cor:vogantempered} plus \cite{CFMMX}*{Conjecture 1 (a), Section 8.3} implies the enhanced Shahidi's conjecture, \cite{LLS}*{Conjecture 1.2}, which is a generalization of \cite{Shahidi}*{Conjecture 9.4}. 
But in fact, one only needs one direction of \cite{CFMMX}*{Conjecture 1 (a), Section 8.3} to get \cite{LLS}*{Conjecture 1.2}, as we now show.

\begin{proposition}\label{cor: main cor}
Let $G$ be a quasi-split classical group. Let $\psi$ be an Arthur parameter. 
If $\Pi^\pure_\psi(G) \subset\Pi_{\phi_\psi}^{\ABVpure}(G)$ for every Arthur parameter $\psi$,
then \cite{LLS}*{Conjecture 1.2} is true: $\Pi^\pure_\psi(G)$ contains a generic representation if and only if $\psi$ is tempered. 
\end{proposition}

\begin{proof}
By Corollary~\ref{cor:vogantempered}, if $\psi$ is tempered, then $\Pi^\pure_\psi(G)=\Pi^\pure_\phi(G)=\Pi_{\phi}^{\ABVpure}(G)$, which contains a generic representation by \cite{GI}*{Theorem B.2}.
 If $\psi$ is not tempered, then $\phi=\phi_\psi$ is not open, by Proposition~\ref{prop: temper=arthur+open}. 
 Thus \cite{GI}*{Theorem B.2} says that $L(s,\phi, \Ad)$ is not regular at $s=1$. Corollary \ref{corollary: classical} implies that $ \Pi_{\phi_\psi}^{\ABVpure}(G)$ does not contain any generic representation. 
 Then the assumption $\Pi_\psi^\pure(G)\subset \Pi_{\phi_\psi}^{\ABVpure}(G)$ implies that $\Pi_\psi^\pure(G) $ does not contain any generic representation.  
\end{proof}

\begin{remark}
The hypothesis $\Pi^\pure_\psi(G)\subset\Pi_{\phi_\psi}^{\ABVpure}(G)$ implies all representations $\pi$ in $\Pi^\pure_\psi(G)$ have the same infinitesimal parameter. 
This is true by \cite{Moeglin}*{Proposition 4.1} for classical groups.
\end{remark}

\begin{remark}
We first became aware of \cite{LLS}*{Conjecture 1.2} at Shahidi's 2021 talk at the \href{https://conferences.cirm-math.fr/2903.html}{2023 CIRM conference Formes automorphes, endoscopie et formule des traces}.
For symplectic and split odd orthogonal groups, a proof was announced in \cite{HLL} and also in several other preprints with intersecting authorship, including \cite{HLLZ}.
\end{remark}

\section{ABV-packets for open parameters}

\subsection{ABV-packets are L-packets for open parameters}

\begin{theorem}\label{theorem:mainopen}
Let $G$ be a connected reductive algebraic group. 
If $\phi$ is an open parameter for $G$ then the ABV-packet for $\phi$ is the pure L-packet for $\phi$:
\[
\Pi^\pure_\phi(G)
=
\Pi^{\ABVpure}_{\phi}(G).
\]
Moreover, the function $\NEvs_\phi \mathcal{P}_\mathfrak{w}$ defined by vanishing cycles coincides with the function $J(\mathfrak{w})$ from the local Langlands correspondence:
\[
\begin{tikzcd}
\Pi^\pure_\phi(G) \arrow[equal]{d}  \arrow{rr}{J(\mathfrak{w})} && \widehat{A_\phi} \arrow[>->]{d}\\
\Pi^{\ABVpure}_{\phi}(G) \arrow{rr}{\NEvs_\phi \circ\mathcal{P}_{\mathfrak{w}}} && \Rep(A_\phi^\ABV).
\end{tikzcd}
\]
\end{theorem}

\begin{proof}
    By Theorem~\ref{theorem:compatibility}, this diagram commutes, so it only remains to show the two equalities above: if $\phi$ is open then $A_\phi^\ABV = A_\phi$ and $\Pi_\phi^{\ABVpure}(G) = \Pi_\phi^\pure(G)$.

    First, recall that $A_\phi^\ABV$ is the equivariant fundamental group of $\Lambda^\mathrm{gen}_{C_\phi}$. 
    Since $\phi$ is open, $C_\phi$ is the open $H_\lambda$-orbit in $V_\lambda$. Thus, $C_\phi^*$ is the closed orbit $\{ 0 \}$ in $V_\lambda^*$, by Proposition~\ref{lemma:open}. 
    Thus, $\Lambda_{C_\phi} =  C_\phi \times C_\phi^*$.
    Now observe that $\Lambda_{C_\phi}\iso C_\phi$ as an $H_\lambda$-space.
    Since this is a single $H_\lambda$-orbit, it follows that $\Lambda_{C_\phi}^\mathrm{gen}$.
    (Note that, in this case, the inclusions $\Lambda_{C_\phi}^\mathrm{gen} \subseteq \Lambda_{C_\phi}^\mathrm{reg} \subseteq \Lambda_{C_\phi}$ are all equalities.)
    Thus, $\Lambda_{C_\phi}^\mathrm{gen}\iso C_{\phi}$. 
    Now it follows that $A^\ABV_\phi$, the equivariant fundamental group of $\Lambda^\mathrm{gen}_{C_\phi}$ is $A_\phi$, the equivariant fundamental group of $C_\phi$:
    \[
    A^\ABV_\phi = A_\phi. 
    \]

    Now suppose, for a contradiction, that $\Pi^{\ABVpure}_\phi(G)$ contains a coronal representation; that is, suppose $\pi \in \Pi^{\ABVpure}_\phi(G)$ and $\pi \not\in \Pi^\pure_\phi(G)$.
    We have $\mathcal{P}_\mathfrak{w}(\pi) = \IC(\underline{\rho}_{C_{\phi_\pi}})$
    and $\NEvs_{\phi} \IC(\underline{\rho}_{C_{\phi_\pi}}) \ne 0$. 
    By \cite{CFMMX}*{}, this implies $C_{\phi_\pi} > C_{\phi}$, with strict inequality since $\pi \not\in \Pi^\pure_\phi(G)$. But since $\phi$ is open, $C_{\phi_\pi} > C_{\phi}$ is not possible. Thus,
    \[
    \Pi_\phi^{\ABVpure}(G) = \Pi_\phi^\pure(G).
    \]
This concludes the proof of Theorem~\ref{theorem:mainopen}.  
\end{proof}

\subsection{Proof of Vogan's conjecture for tempered parameters}

As shown in \cite{CFMMX}*{Section 3.11}, it follows from Arthur's main local result \cite{Arthur:book}*{Theorem 1.5.1} that pure A-packets come equipped with a canonical function $\Pi^\pure_\psi(G) \to \widehat{A_\psi}$; since this function depends on a Whittaker datum, we denote that function here by
\[
A(\mathfrak{w}) : \Pi^\pure_\psi(G) \to \widehat{A_\psi}.
\]
Vogan's conjecture on A-packets \cite{CFMMX}*{Conjecture 1, Section 8.3} says that if $G$ is a quasi-split classical group over a $p$-adic field and $\psi$ is an Arthur parameter for $G(F)$ then 
$
\Pi^\pure_\psi(G) = \Pi^{\ABVpure}_{\phi_\psi}(G)
$
and the function $\NEvs_{\phi_\psi} \circ\mathcal{P}_w$ in Equation \eqref{eq: Arthur map} is identical to that defined by Arthur. 
In particular, this includes the prediction that the function $\NEvs_\phi \circ\mathcal{P}_{\mathfrak{w}}$ produces only irreducible representations.
We now prove this conjecture for tempered parameters.
   An Arthur parameter $\psi$ is said to be tempered if its restriction to the Deligne part of the Weil-Deligne group is trivial. In this case the corresponding Langlands parameter is tempered.
   
\begin{corollary}[Vogan's conjecture on A-packets, for tempered parameters]\label{cor:vogantempered}
Let $G$ be a quasi-split classical group over $F$ (or any group for which we know Arthur's conjecture).
Let $\psi$ be a tempered Arthur parameter for $G$. 
Then the pure Arthur packet for $\psi$ is equal to the ABV-packet for the Langlands parameter $\phi_\psi$:
\[
\Pi^\pure_\psi(G)
=
\Pi^{\ABVpure}_{\phi_\psi}(G).
\]
Moreover, $A_\psi = A^\ABV_{\phi_\psi}$ and the following diagram commutes.
\[
\begin{tikzcd}
\Pi^\pure_\psi(G) \arrow[equal]{d}  \arrow{rr}{A(\mathfrak{w})} && \widehat{A_\psi} \arrow[>->]{d}\\
\Pi^{\ABVpure}_{\phi_\psi}(G) \arrow{rr}{\NEvs_{\phi_\psi} \circ\mathcal{P}_{\mathfrak{w}}} && \Rep(A^\ABV_{\phi_\psi})
\end{tikzcd}
\]
In particular, the function $\NEvs_{\phi_\psi} \circ\mathcal{P}_{\mathfrak{w}}$ produces only irreducible representations.
\end{corollary}

\begin{proof}
Let $\psi$ be an Arthur parameter for $G(F)$ such that the corresponding Langlands parameter $\phi_\psi$ is tempered. Then, by Proposition~\ref{prop: temper=arthur+open}, $\phi_\psi$ is open.
By Theorem~\ref{theorem:mainopen}, $A^\ABV_{\phi_\psi} = A_{\phi_\psi}$ and 
\[
\begin{tikzcd}
\Pi^\pure_{\phi_\psi}(G) \arrow[equal]{d}  \arrow{rr}{J(\mathfrak{w})} && \widehat{A_{\phi_\psi}} \arrow[>->]{d} \\
\Pi^{\ABVpure}_{\phi_\psi}(G) \arrow{rr}{\NEvs_{\phi_\psi} \circ\mathcal{P}_{\mathfrak{w}}} && \Rep(A^\ABV_{\phi_\psi})
\end{tikzcd}
\]
commutes. 
On the other hand, since $\psi$ is tempered, 
it follows from \cite{Arthur:book}*{Theorem 1.5.1(b)} that $A_\psi = A_{\phi_\psi}$ and
\[
\begin{tikzcd}
\Pi^\pure_{\phi_\psi}(G) \arrow[equal]{d}  \arrow{rr}{J(\mathfrak{w})} && \widehat{A_{\phi_\psi}} \arrow[equal]{d} \\
\Pi^\pure_\psi(G)  \arrow{rr}{A(\mathfrak{w})} && \widehat{A_\psi} 
\end{tikzcd}
\]
commutes. 
This completes the proof of Corollary~\ref{cor:vogantempered}.
\end{proof}

\appendix

\section{Discrete and tempered parameters are open}\label{sec:discrete}

We have seen in Proposition~\ref{prop: temper=arthur+open} that a Langlands parameter $\phi$ is tempered if and only if it is open and of Arthur type. 
Of course, this implies that if $\phi$ is tempered then it is open. 
In this section we give an alternative proof that tempered implies open using different techniques that illustrate the connection between open parameters and distinguished nilpotent orbits. In fact, we first give a proof that if $\pi$ is a square-integrable representation its Langlands parameter $\phi_\pi$ is open.

\subsection{Parameters for square-integrable representations are open}\label{ssec:discrete-open}

In this section we prove that the Langlands parameters for square integrable representations are open parameters:

\begin{proposition} \label{discrete}
    Let $G$ be a connected reductive algebraic group over $F$ and let $\pi$ be a square-integrable representation of $G$. 
    Let $\phi_\pi$ be a Langlands parameter for $\pi$, and $\lambda$ its infinitesimal parameter.
    Then $\phi_\pi$ is an open parameter.
\end{proposition}

 Our proof (given at the end of this section)  relies on certain results presented  in \cite{carter}, as well as a Theorem of Heiermann in \cite{heiermannorbit}. 
 Carter's book \cite{carter}*{Sections 5.5 to 5.8} reduces its presentation to the case of $\mathcal{G}$ a (complex) reductive group which is adjoint semi-simple.
 Definition \ref{q-dist} supposes similar assumptions, however, the two results of Carter crucial to our argumentation do not explicitly. Therefore, the proposition we prove hold for any reductive $p$-adic group, with no restriction on its dual complex group $\dualgroup{G}$. 
 
 We denote our complex reductive group by $\mathcal{G}$, $Z_\mathcal{G}$ its center, and choose a $\mathcal{B}$ a Borel subgroup $\mathcal{B} \supset \mathcal{T}$, with $\mathcal{B}= \mathcal{T}\mathcal{U}$ its unipotent radical, and we denote their respective Lie algebras $\g$, $\T$, $\U$. Let $\Phi = \Phi(\mathcal{G},\mathcal{T})$ be the set of roots of $\mathcal{T}$ on $\g$, and let $\Delta$ be the basis of $\Phi$ corresponding to $\mathcal{B}$.

We will be particularly interested in distinguished nilpotent elements (orbits). For that reason, we recall here some pre-requisites needed to understand the notations used in the proof, as well as in the following section. 

\begin{definition}[Distinguished nilpotent element]
An element $N$ of $\g$ is distinguished if and only if exp($N$) is not contained in any proper Levi of $\mathcal{G}$.
\end{definition}

Given a non-zero nilpotent element $N \in \g$,
let $\{e,h, f\}$ denote the standard basis of the $\Sl_2$ Lie algebra. One may regard $\g$ as an $\Sl_2$-module via a Jacobson-Morozov Lie algebra homomorphism: $\phi: \Sl_2 \rightarrow \g$ which, modulo conjugation, satisfies $\phi(e) = N \in \U$ and $\phi(h) = \gamma$ is in the dominant chamber of $\T$. 
It is a theorem of Dynkin \cite{Dynkin}*{Theorem 8.3}  (see also \cite{Humphreys}*{Proposition 7.6} that, in this case, $\langle \alpha, \gamma \rangle \in \{0,1,2\}$ for any $\alpha \in \Delta$. Such an element $\gamma$ will be called a weighted Dynkin diagram. By  $\Sl_2$  theory, we obtain a grading $\g= \oplus \g(i)$ where $\g(i)$ is $\left\{x \in \g \tq \ad(\gamma)x= ix\right\}$ and $N \in \g(2)$. We also have:
$$\p = \p(\gamma) = \oplus_{i \geq 0}\g(i) $$
$$ \U =  \oplus_{i > 0}\g(i) $$
$$ \lev= \oplus_{i = 0}\g(i)$$

The Lie subalgebra $\p$ contains $\bo$ and is thus a parabolic subalgebra whose Levi decomposition is $\p = \lev\U$. On the other hand, starting with a standard parabolic
subalgebra $\p_J$ , for $J \subset \Delta$, let $\eta_J:\Phi \rightarrow Z_\mathcal{G}$ be a function defined on $\Delta$ as twice the indicator function of $J$ and then extended linearly to all roots. Then, one also gets a grading $\g= \otimes_{i \geq 0} \g_J(i)$ by declaring $\g_J(0) = \T\oplus \sum_{\eta_J(\alpha)=0}\g_\alpha$ and $\g_J(i) = \sum_{\eta_J(\alpha)=i}\g_\alpha$.

\begin{equation} \label{parabolic}
\p_J = \oplus_{i \geq 0}g(i), ~~\U_J=  \oplus_{i > 0}g(i), ~~ \hbox{and} ~~\lev_J= \oplus_{i = 0}g(i)
\end{equation} 

We have (cf. \cite{carter}, Corollaries 5.7.5, 5.8.3, 5.8.9 and Proposition 5.7.6): 
\begin{proposition}
The standard parabolic subalgebra $\p_J$ is distinguished if and only if $\dim\g_J(0) = \dim\g_J(2)$. In this case, if $N$ is any element in the unique open orbit of the parabolic subgroup $P_J$ on its nilpotent radical $\U_J$, then the parabolic subalgebra associated to N as above equals to $\p_J$ .
\end{proposition}

\begin{proposition}
A nilpotent element $N \in \g$  is distinguished if and only if $\dim\g(0) = \dim\g(2)$. 
\end{proposition}

Let us consider an (unramified) infinitesimal parameter $\lambda$ as defined in subsection \ref{infinitesimaldef}. We can observe that the $2$-eigenspace $\g(2)$ is precisely the Vogan variety by taking the log of $\lambda$ and comparing it to $\gamma$, up to some normalization factor.
To take into account a semi-simple element of the group itself (such as $\lambda$) instead of its Lie algebra, we will follow Heiermann in defining $q$-distinguished semi-simple element $\lambda$ in $\mathcal{G}$. 

\begin{definition}[Definition 4.5 in \cite{heiermannorbit}] \label{q-dist}
Let $\lambda$ be a semi-simple element in $\mathcal{G}$. Let $\g^{\small{\hbox{der}}}$ be the Lie algebra of the derived subgroup of $\mathcal{G}$. Let us consider an eigenvalue $\Lambda$ for the $\Ad(\lambda)$ action on $\g^{\small{\hbox{der}}}$, and denote $\g_\lambda(\Lambda)$ the corresponding eigenspace. Then $\lambda$ is $q$-distinguished if $\dim\g_\lambda(q) = \dim\g_\lambda(1)$.
\end{definition}

Note the parallel with the definition of distinguished nilpotent element. This is not coincidental, as stated in Heiermann's:

\begin{proposition}[\cite{heiermannorbit}*{Prop. 4.7}] \label{4.7}
Let $\lambda$ the hyperbolic part of a $q$-distinguished element of $G$. Let us fix a maximal torus $T$ of $\mathcal{G}$ containing $\lambda$ and a set of simple roots $\Delta$ with respect to $T$, for which $\lambda$ is positive. Then any root of $\Delta$ satisfies $\alpha(\lambda)=1$ or $\alpha(\lambda)=q$. There exists a distinguished nilpotent element $N$ in Lie($G$) so that $\Ad(\lambda)N= qN$.
\end{proposition}

We now want to prove that a discrete Langlands parameter is open. Our strategy will be to claim that the infinitesimal parameter $\lambda$ attached to a discrete Langlands parameter is $q$-distinguished, using \cite{heiermannorbit}*{Theorem 6.3} recalled below. Then Proposition \ref{4.7} allows us to attach a distinguished nilpotent element to it, which, by definition (see Proposition \ref{4.7}), lies in the Vogan variety. Then two propositions of \cite{carter} give us the conclusion. 

\begin{theorem}\cite{heiermannorbit}*{Theorem 6.3} [(see also 6.4)] \label{heiermanndiscrete}
Let $\pi$ be a square-integrable representation of $G$ which is a subquotient in $I_P^G(\sigma\otimes \chi_\lambda)$, and assume $\sigma$ is a cuspidal representation of $M$ such that its Langlands parameter $\phi_\sigma$ is discrete and that there exists a minimal semi-standard Levi in ${}^{L}M$ that contains $\phi_\sigma|_{W_F}$ and $\phi_\sigma(d_w|_{|w|=q})$.
Then, there exists a $q$-distinguished element $s_\lambda$ of the connected centralizer of $\phi_\sigma|_{W_F}$ in ${}^{L}G$, denoted $\hat{M}^\sigma$, whose compact part is trivial and by Proposition \ref{4.7}, we can associate a distinguished nilpotent element $N$ to this $s_\lambda$ (such that $(s_\lambda, N_{\lambda, \sigma})$ is a $L^2$-pair in the terminology of Lusztig).   
Then we set $\phi_\pi(w, h) = \phi_\sigma(w,1)\phi_{\sigma, \lambda}(h)$ where $\phi_{\sigma, \lambda}: SL_2(\C) \rightarrow \hat{M}^\sigma$:
$$ \phi_{\sigma,\lambda}\begin{pmatrix} 1 & 1 \\ 0 & 1 \end{pmatrix}=\exp(N_{\lambda,\sigma}) ~~ \hbox{and} ~~ \exp\begin{pmatrix} q^{1/2} & 0 \\ 0 & q^{-1/2} \end{pmatrix} = s_\lambda $$
This is a well-defined map from $W_F\times \SL_2(\C)$ to ${}^{L}G$, which is admissible and discrete, and uniquely determined by $\pi$. 
\end{theorem}

With some minor additional hypothesis (in 6.4 of \cite{heiermannorbit}) to insure the parameter $\phi_\sigma$ is necessarily associated to a cuspidal representation $\sigma$, Heiermann shows that any discrete admissible Langlands parameter from $W_F\times \SL_2(\C)$ to ${}^{L}G$ is equivalent to the above (i.e., $\phi_\pi$ for $\pi$ a square integrable representation).

\begin{proof}[Proof of Proposition \ref{discrete}]
We use the definition of open Langlands parameter as given in Definition \ref{def:open}. 
By Heiermann's Theorem \ref{heiermanndiscrete}, we have that $\lambda$, denoted $s_\lambda$ in Heiermann's notation, is $q$-distinguished in $\hat{M}^\sigma$, the connected centralizer of $\phi_\sigma|_{W_F}$ in ${}^{L}G$, which corresponds to $\dualgroup{G_\lambda}$ in the notations of \cite{CFMMX}*{Section 5.3}. 
By Proposition \ref{4.7}, the nilpotent element $N_{\sigma, \lambda}$ (see also Prop. 4.9 in \cite{heiermannorbit}) is a distinguished nilpotent element in the Lie algebra of $\hat{M}^\sigma$ and is uniquely (up to a scalar) determined by $s_\lambda:= \lambda$ (Prop. 4.9 in \cite{heiermannorbit}). 
Further by \cite{heiermannorbit}*{Prop. 4.9} again, as well as the few first words of Section 5.6 in \cite{carter}, we note that the centralizer of $s_\lambda$ in $\dualgroup{G_\lambda}:= \hat{M}^\sigma$ is a Levi subgroup, $L_J$ whose Lie algebra is $\hat{\m}_J^\sigma(0) = \lev_J$. Now, we can apply Propositions 5.8.4 and 5.8.7 in \cite{carter} which state that given a distinguished nilpotent with associated parabolic $P_J$, the $L_J$-orbit of $e$ is open and dense in $\g_J(2)$.
Applying these to $N_{\sigma, \lambda}$, we obtain that the $L_J$-orbit of $N_{\sigma, \lambda}$ is open and dense in $\hat{\m}_J^\sigma(2)$, which, by definition, is $V_{\lambda_{\hbox{nr}}}$. Finally, we use \cite{CFMMX}*{Lemma 5.3}, i.e., the fact that $V_{\lambda_{\hbox{nr}}} = V_\lambda$, and $H_{\lambda_{\hbox{nr}}} =H_{\lambda}^0= L_J$, to rephrase this as: The $H_\lambda^0$-orbit of $N_{\sigma, \lambda}:= x_\phi$ is open and dense in $V_\lambda$. If $H_\lambda$ is disconnected, observe that since the $H_\lambda^0$-orbit is the largest, open one, the $H_\lambda$-one will necessary be open too. 
\end{proof}

\subsection{Tempered parameters are open}\label{ssec:tempered-open}



The next result requires the following desiderata for the local Langlands correspondence given in \cite{abps} based on \cite{Borel:Corvallis}. 
To express this, we need to introduce enhanced Langlands parameters, $\Phi^e(G)$ of a reductive $p$-adic group $G$, in the sense that we add an irreducible complex representation $\rho$ of the $S$-group of $\phi$: $S_\phi =Z_{\hat{G}}(\phi)\slash Z_{\hat{G}}(\phi)^0Z(\hat{G})^{\Gamma_F}$. Recall also that tempered representations (see \cite{wald}*{Proposition III.4.1}) of a group $G$ are subquotients (equivalently, direct summands) of the normalized parabolic induction of an essentially discrete series representation of a Levi subgroup. Let $M$ be a Levi subgroup of $G$. 

\leavevmode
\begin{itemize}
\item If $(\phi_M , \varrho_M) \in \Phi^e(M)$ is bounded then
\[
\left\{\pi_{\phi, \rho}: \phi = \phi_M ~~ \hbox{composed with} ~~ {}^{L}\eta: {}^{L}M \rightarrow {}^{L}G, \rho|_{S_{\phi_M}}
~ \hbox{contains} ~ \varrho_M \right\} 
\]
equals the set of all irreducible constituents of the parabolically induced representation $I_P^G(\pi_{\phi_M,\varrho_M})$.
\item Furthermore if $\phi_M$ is discrete but not necessarily bounded then the set above is the set of Langlands constituents of $I_P^G(\pi_{\phi_M,\varrho_M})$.
\end{itemize}

\begin{proposition}
Let $G$ be a connected reductive group over $F$. Assume the desiderata above for $G$. Let $\phi_\pi$ be a tempered Langlands parameter for $G$ and let $\lambda_\phi$ be its infinitesimal parameter. Then $\phi_\pi$ is open in $V_\lambda$. 
\end{proposition}

\begin{proof}
Any tempered Langlands parameter can be understood as ${}^{L}\eta (\phi_M)$ for a discrete unbounded parameter of a Levi $M$ of $G$, such that the inducing parabolic, as used in the desiderata above, is $P=MU$. 
Then by the proof of the Proposition \ref{discrete}, we have that $\lambda_{\phi_M}$ is $q$-distinguished and $N_{\sigma, \lambda_{\phi_M}}$ is a nilpotent distinguished element in $  \dualgroup{\m}$. 
Now, given this distinguished nilpotent element $N_{\sigma, \lambda_{\phi_M}} \in \dualgroup{\m}$ we consider it as an element in $\dualgroup{\g}$. We cannot argue by showing that $\lambda_{\phi_M}$ is $q$-distinguished in $\dualgroup{\g}$, as this is not true in general. Let us denote $N_{\sigma, \lambda_{\phi_M}}:=e$ and let $L_J$ be the centralizer in $\dualgroup{M}$ of $\lambda_{\phi_M}$, also denoted $H_{\lambda_{\phi_M}}$. Up to conjugation by a Weyl group element, the semi-simple element $\lambda$ ($\lambda_{\phi_M}$ now seen as an element in $\dualgroup{\g}$) can be taken to be in the positive Weyl chamber, without impacting the geometrical argument we are using below. This semi-simple element defines uniquely a parabolic subgroup $P_0$ (see the beginning of this appendix) of $\dualgroup{G}$ and a Levi, $M_0 \supseteq L_J$, where $M_0 = H_\lambda^0$. 
The element $e$ is a nilpotent element in $\dualgroup{\g}$, and $(\lambda, e)$ naturally forms part of a Jacobson-Morozov triple in $\dualgroup{\g}$. We now wish to show that the $M_0$-orbit of $e$ is open in $\dualgroup{\g}(2)$. This follows from the argument presented in \cite{rangarao}*{Section 3}: There it is shown that $\Ad(M_0)e = \left\{\Ad(m).e: m \in M_0\right\}$ is open in the Hausdorff topology of $\dualgroup{\g}(2)$. If $H_\lambda$ is disconnected, observe that since the $H_\lambda^0$-orbit is the largest, open one, the $H_\lambda$-one will necessary be open too, which completes the proof.
\end{proof}

\subsection{Example with the case of classical groups}
\label{appclassical}

Following from the observation made in the above subsection that discrete Langlands parameters have $q$-distinguished infinitesimal parameters, we deduce that the restriction to $\SL_2(\C)$ of the Langlands parameter $\phi$ of an irreducible discrete series representation of a classical group whose root system is of type $B,C,D$ is necessarily written as the direct sum of $\nu_i$ with \textbf{distinct} odd or even integers $i$, where $\nu_i$ stands for the irreducible representation of $\SL_2$ of dimension $i$: $\hbox{Sym}^{i-1}$ (see also \cite{GR}). Indeed, evaluating each of the $\nu_i= \hbox{Sym}^{i-1}$ where $i$ runs over the integers of a given partition (characterizing the distinguished nilpotent orbit) at $d_w$ gives the infinitesimal parameter $\lambda(w)$
\footnote{In case the root system is of type $A$, the restriction to $\SL_2(\C)$ of the Langlands parameter is a unique $\nu_i$ of maximal dimension}. The partition determines the shape of the endoscopic elliptic subgroup ${}^{L}M_E$ of ${}^{L}G$ the discrete parameter factors through. Consider the Jordan form of $x_\phi= \log\phi(1, e)$ in ${}^{L}M_E$, it is a distinguished nilpotent element whose orbit under $H_{\lambda_{{}^{L}M_E}}$ is trivially open and dense in $V_{\lambda_{{}^{L}M_E}}$. Indeed, we can \textit{see} that $x_\phi$ fills in all the non-zero entries in $V_{\lambda_{{}^{L}M_E}}$. We can now use a permutation matrix (i.e a Weyl group element) which allows us to reorganize the $q$-exponents in a decreasing fashion to consider $\lambda_{{}^{L}M_E}$ in ${}^{L}G$ rather than in ${}^{L}M_E$. Said differently, the variety $V_\lambda$ is usually more easily described when the infinitesimal parameter is a decreasing sequence of $q$-exponents. 

Let us denote $\lambda$ the result after permutation. The permutation matrix also acts on $x_\phi$ in a way that each non-zero entry at the intersection of a line with diagonal entry $q^i$ and a column with diagonal entry $q^j$ will follow those diagonal entries, as they are moved along the diagonal, to remain at the intersection of their line and column. 

In \cite{CFMMX}*{10.2.1} and in \cite{benesh}, the Vogan variety, for classical groups, has been described as a product of Hom-spaces $\Hom(E_{q^i}, E_{q^{i+1}})$ where $E_{q^i}$ stands for the $q^i$-eigenspace in $V_\lambda$. Here, we also use the description of the open orbit in terms of maximal ranks of maps between the eigenspaces $E_{q^i}$ given in \cite{CFMMX}*{10.2.1}. 

We denote the map ${}^{L}\eta : {}^{L}M_E \rightarrow {}^{L}G$ and aim to show that ${}^{L}\eta(x_\phi)$ will remain in the open orbit of $V_\lambda$.  The key observation is that the non-zero entries in $x_\phi$, which also stand for maps between $E_{q^i}$-eigenspaces in $\prod_{i=1}^j V_{\lambda_i}$, will remain maps between those same eigenspaces in $V_\lambda$ after permutation, hence preserving the ranks of such maps and their compositions. 

Let us formulate this fact in a lemma, starting with a definition.
%
Let $K$ be an algebraically closed field. Let $\lambda$ be a diagonalizable linear endomorphism on a finite dimensional $K$-vector space, $E \cong K^n$. After selecting a basis where $\lambda$ is diagonal, we may write $\lambda$ a as block diagonal matrix where each block is given by integers or half-integers $q$-exponents, centered around 0. Let us denote $E_q^i$ the $q^i$-eigenspace, for any such exponent $i$. A linear endomorphism $S_\lambda$ is said to be \textit{of degree +1} if for any unique $v \in E_{q^i}$ and unique $w \in E_{q^{i+1}}$, there is a unique map $S_\lambda(v)=w$. Further, if the dimensions of $E_q^i$ and $E_q^{i+1}$ differ, then the rank of the map between them needs to be $\min(\dim(E_{q^i}), \dim(E_{q^{i+1}}))$. 

\begin{lemma} \label{endo} 
Let $K$ be an algebraically closed field. Let $\lambda$ be a diagonalizable linear endomorphism on a finite dimensional $K$-vector space, $E \cong K^n$. Viewed as a linear endomorphism of $K^n$ and selecting a basis $\mathcal{B}$ of $K^n$ where $\lambda$ is diagonal, we may write $\lambda$ as block diagonal matrix where each block is given by integers or half-integers $q$-exponents, centered around 0. We take $S_\lambda$ a degree +1 (as defined above). Then the linear endomorphism $S_\lambda$, viewed as a linear endomorphism of $K^n$, is invariant under change of basis, in particular under reordering of $\mathcal{B}$.
\end{lemma}

\begin{proof}
It is well-known that linear endomorphisms are invariant under change of basis. 
\end{proof}

The Lemma \ref{endo} can also be used in the context of tempered parameters, using the Borel desiderata for the Langlands correspondence as recalled above. Then, we start from a discrete Langlands parameter $\phi_M$, which we have proven is open, and consider the embedding ${}^{L}\eta(\phi_M)$. The parameter $\phi_M$ decomposes in a $\GL$ part and a classical part, each contributing to its infinitesimal character $\lambda_{\phi_M}= (\lambda_{\GL}, \lambda_c)$. Whether the segments (the sequence of $q$-exponents in each part of $\lambda$) are nested or disjoint, Lemma \ref{endo} shows that, when reorganizing $\lambda$ so that it is a decreasing sequence of (half)-integers, we simply \textit{add} multiplicities of $q$-exponents, and add the ranks of the $\Hom(E_{q^i}, E_{q^{i-1}})$'s (all equal to 1 before reorganization), as well as their compositions, so that the ranks remain maximal with respect to multiplicities, hence proving openness. 

\begin{proposition}
Let $\phi$ be a Langlands parameter of a classical split group $G$ or any of its pure inner forms, factoring through some subgroup ${}^{L}H$ of ${}^{L}G$, and with infinitesimal parameter $\lambda$. If $\phi$ is such that $x_{\phi|_{{}^{L}H}}$ lies in the open orbit of the Vogan variety $V_{\lambda|_{{}^{L}H}}$, then it lies in the open orbit in $V_{\lambda}$. 
Further, this implies that parabolic induction preserves the openness of a Langlands parameter. 
\end{proposition}

\begin{proof}
A direct consequence of the Lemma \ref{endo}, and following the same argumentation as to prove discrete and tempered parameters are open just above. We use the description of the Vogan variety for classical groups in terms of product of homomorphism spaces given in \cite{CFMMX}*{Chapter 10}. 
\end{proof}

\end{document}